\documentclass[11pt]{amsart}

\usepackage{graphicx}%
\usepackage{multirow}%
\usepackage{amsmath,amssymb,amsfonts}%
\usepackage{amsthm}%
\usepackage{mathrsfs}%
\usepackage[title]{appendix}%
\usepackage{xcolor}%
\usepackage{textcomp}%
\usepackage{manyfoot}%
\usepackage{booktabs}%
\usepackage{algorithm}%
\usepackage{algorithmicx}%
\usepackage{algpseudocode}%
\usepackage{listings}%
\usepackage[colorlinks=true, linkcolor=red, urlcolor=blue, citecolor=blue]{hyperref}

\theoremstyle{definition}
\newtheorem{theorem}{Theorem}[section]
\newtheorem{proposition}[theorem]{Proposition}
\newtheorem{lemma}[theorem]{Lemma}
\newtheorem{corollary}[theorem]{Corollary}
\theoremstyle{definition}
\newtheorem{definition}[theorem]{Definition}
\newtheorem{example}[theorem]{Example}
\theoremstyle{remark}
\newtheorem{remark}[theorem]{Remark}

\allowdisplaybreaks[1]
\usepackage[left=2.5cm, right=2.5cm, top=3cm, bottom=3cm]{geometry}

\title[$q$-Berezin Range of Operators in Hardy Space]{$q$-Berezin Range of Operators in Hardy Space}

\author{Debarati Bhattacharya} 
\address{Department of Mathematics, Indian Institute of Technology Bhilai, Durg, 491002, Chhattisgarh, India}
\email{debaratib@iitbhilai.ac.in}

\author{Arnab Patra}
\address{Department of Mathematics, Indian Institute of Technology Bhilai, Durg, 491002, Chhattisgarh, India}
\email{arnabp@iitbhilai.ac.in}

\subjclass[2020]{Primary 47B32, 47B38; Secondary 47A05, 30H10}
\keywords{$q$-Berezin range, Convexity, Reproducing kernel Hilbert space, Toeplitz operator, Composition operator.}

\begin{document}

\begin{abstract}
This article investigates the geometric structure of the q-Berezin range of operators on the Hardy space, with a particular emphasis on convexity. Convexity results are established for several operator classes, including Toeplitz, weighted shift, and certain composition operators. The study highlights the relationship between the $q$-Berezin range and the classical Berezin range.
\end{abstract}

\maketitle


\section{Introduction}\label{sec1}

Let $\mathcal{H}$ be a complex Hilbert space equipped with inner product $\langle \cdot,\cdot\rangle_{\mathcal{H}}$ and the induced norm $\|\cdot\|_{\mathcal{H}}$. We denote by $\mathscr{B}(\mathcal{H})$ the $C^*$-algebra of all bounded linear operators on $\mathcal{H}$. The open unit disc and the circle of radius $r$ centered at the origin in the complex plane are denoted by $\mathbb{D}$ and $\mathbb{T}_r$, respectively, with $\mathbb{T}_1=\mathbb{T}$.

A \emph{reproducing kernel Hilbert space} (RKHS) on a set $\Omega$ is a Hilbert space $\mathscr{H}$ of complex-valued functions such that, for each $x\in\Omega$, the evaluation functional $E_x:\mathscr{H}\to\mathbb{C}$ defined by $E_x(f)=f(x)$ is bounded. By the Riesz representation theorem, for every $x\in\Omega$ there exists a unique element $k_x\in\mathscr{H}$ satisfying
\[
f(x)=\langle f,k_x\rangle_{\mathscr{H}}, \qquad f\in\mathscr{H}.
\]
The function $k_x$ is called the reproducing kernel at $x$, and its normalized version is denoted by $\hat{k}_x = k_x/\|k_x\|_{\mathscr{H}}$. It follows that $k_x(y)=\langle k_x,k_y\rangle_{\mathscr{H}}$ and $\|E_x\|_{\mathrm{op}}^2=\|k_x\|_{\mathscr{H}}^2$. Classical examples of RKHS include the Hardy, Bergman, Dirichlet, and Fock spaces (see \cite{paulsen2016introduction}).

For $T\in\mathscr{B}(\mathscr{H})$, the \emph{Berezin transform} of $T$ is the function $\tilde{T}:\Omega\to\mathbb{C}$ defined by
\[
\tilde{T}(w)=\langle T\hat{k}_w,\hat{k}_w\rangle_{\mathscr{H}}, \qquad w\in\Omega,
\]
introduced by Berezin \cite{berezin1972covariant}. The associated \emph{Berezin range} and \emph{Berezin number}, introduced by Karaev \cite{karaev2006berezin}, are given respectively by
\[
\mathrm{Ber}(T)=\{\langle T\hat{k}_w,\hat{k}_w\rangle_{\mathscr{H}}: w\in\Omega\}, 
\qquad
\mathrm{ber}(T)=\sup_{w\in\Omega}|\tilde{T}(w)|.
\]
These quantities satisfy $\mathrm{Ber}(T)\subseteq W(T)$ and $\mathrm{ber}(T)\leq w(T)\leq \|T\|_{\mathrm{op}}$, where
\[
W(T)=\{\langle Tx,x\rangle_{\mathcal{H}}:\|x\|_{\mathcal{H}}=1\}, \qquad
w(T)=\sup\{|\langle Tx,x\rangle_{\mathcal{H}}|:\|x\|_{\mathcal{H}}=1\}
\]
denote the numerical range and numerical radius of $T$, respectively. While the numerical range is always convex by the Toeplitz--Hausdorff theorem \cite{gau2021numerical}, the Berezin range need not be convex. The geometric and convexity properties of $\mathrm{Ber}(T)$ have attracted considerable attention in recent years; see, for instance, \cite{augustine2023composition, augustine2024berezin, cowen2022convexity, karaev2013reproducing, maity2025convexity}. Convexity questions for Berezin ranges of Toeplitz operators with harmonic symbols have been studied in \cite{sen2025berezin}, and various estimates for Berezin numbers have been obtained in \cite{bhunia2023new, garayev2023weighted, tapdigoglu2024some, zamani2024berezin}.

The \emph{Hardy space} $H^2(\mathbb{D})$ consists of all analytic functions
\[
H^2(\mathbb{D})=\left\{ f(z)=\sum_{n=0}^{\infty} a_n z^n \in \mathrm{Hol}(\mathbb{D}) : \sum_{n=0}^{\infty} |a_n|^2 < \infty \right\}.
\]
It is a RKHS with reproducing kernel
\[
k_w(z)=\frac{1}{1-\overline{w}z}=\sum_{n=0}^{\infty}\overline{w}^n z^n, \qquad z\in\mathbb{D},
\]
and $\|k_w\|^2 = (1-|w|^2)^{-1}$. The inner product and norm on $H^2(\mathbb{D})$ are given by
\[
\left\langle \sum_{n=0}^\infty a_n z^n, \sum_{n=0}^\infty b_n z^n \right\rangle = \sum_{n=0}^{\infty} a_n \overline{b_n}, 
\qquad
\|f\|^2 = \sum_{n=0}^{\infty} |a_n|^2,
\]
making $H^2(\mathbb{D})$ isometrically isomorphic to $\ell^2$.

Recently, Stojiljkovi\'{c} et al. \cite{stojiljkovic25} introduced the notion of the \emph{$q$-Berezin range} and \emph{$q$-Berezin number} for $q\in(0,1]$. For $T\in\mathscr{B}(\mathscr{H})$, these are defined by
\[
\mathrm{Ber}_q(T)=\{\langle T\hat{k}_{w_1},\hat{k}_{w_2}\rangle_{\mathscr{H}} : w_1,w_2\in\Omega,\ \langle \hat{k}_{w_1},\hat{k}_{w_2}\rangle_{\mathscr{H}}=q\},
\]
\[
\mathrm{ber}_q(T)=\sup\{ |z| : z\in \mathrm{Ber}_q(T)\}.
\]
It follows directly from the definition that $\mathrm{Ber}_q(T)$ is a nonempty bounded subset of $\mathbb{C}$, satisfying
\[
\mathrm{Ber}_q(aT+bI)=a\,\mathrm{Ber}_q(T)+bq, \qquad
\mathrm{Ber}_q(T^*)=\mathrm{Ber}_q(T)^*,
\]
for all complex scalars $a,b$. Moreover, $\mathrm{Ber}_1(T)=\mathrm{Ber}(T)$ and $\mathrm{ber}_q(T)\leq \|T\|_{\mathrm{op}}$.

It is straightforward to verify that $\mathrm{Ber}_q(T) \subseteq W_q(T)$, where $W_q(T)$ denotes the $q$-numerical range of $T$, introduced by Marcus and Andresen \cite{marcus1977constrained}. For $|q|\leq 1$, the \emph{$q$-numerical range} is defined by
\[
W_q(T)=\{\langle Tx,y\rangle_{\mathcal{H}} : \|x\|_{\mathcal{H}}=\|y\|_{\mathcal{H}}=1,\ \langle x,y\rangle_{\mathcal{H}}=q\},
\]
and the associated \emph{$q$-numerical radius} is given by
\[
w_q(T)=\sup\{|\langle Tx,y\rangle_{\mathcal{H}}| : \|x\|_{\mathcal{H}}=\|y\|_{\mathcal{H}}=1,\ \langle x,y\rangle_{\mathcal{H}}=q\}.
\]
It is known that $W_q(T)$ is convex (Tsing \cite{tsing1984constrained}), bounded, nonempty, and unitarily invariant. 

\medskip

Despite the recent introduction of the $q$-Berezin range, its geometric structure and convexity properties have not yet been systematically studied. The primary objective of this paper is to initiate such an investigation for bounded linear operators on the Hardy space $H^2(\mathbb{D})$.

The paper is organized as follows. In Section~\ref{sec2}, we present a complete characterization of normalized reproducing kernels satisfying $\langle \hat{k}_{w_1}, \hat{k}_{w_2} \rangle = q$ for $0<q\le 1$. This characterization is important to study the $q$-Berezin range in some cases. In Section~\ref{sec3}, we establish the convexity of the $q$-Berezin range for several classes of operators. In particular, we analyze Toeplitz operators using a $q$-Poisson kernel approach and describe the structure of their $q$-Berezin ranges for both harmonic and analytic symbols. We also examine the $q$-Berezin range of weighted shift operators and investigate the convexity of the $q$-Berezin range for certain classes of composition operators. Finally, Section~\ref{sec4} contains some concluding remarks.

\section{Foundation and geometric structures of $q$-Berezin range}\label{sec2}

We first analyze the geometric structure of pairs of normalized reproducing kernels satisfying the constraint $\langle\hat{k}_{w_1},\hat{k}_{w_2}\rangle=q,$ which plays a central role in the definition. Using this characterization, we develop structural properties of the $q$-Berezin range.
	\begin{lemma}
		For any $w_1,w_2\in\mathbb{D}$, $\langle\hat{k}_{w_1},\hat{k}_{w_2}\rangle\neq0$.
	\end{lemma}
	\begin{proof}
		Let $w_1,w_2\in\mathbb{D}$. Then 
		\[ \langle\hat{k}_{w_1},\hat{k}_{w_2}\rangle = \dfrac{1}{\|k_{w_1}\|\|k_{w_2}\|}\langle k_{w_1},k_{w_2}\rangle = \frac{k_{w_1}(w_2)}{\|k_{w_1}\|\|k_{w_2}\|} = \dfrac{\sqrt{(1-|w_1|^2)(1-|w_2|^2)}}{1-\overline {w_1}w_2}. \]
        Since $w_1,  w_2 \in \mathbb{D},$ $\langle\hat{k}_{w_1},\hat{k}_{w_2}\rangle \neq 0.$
	\end{proof}
    Let $q\in(0,1]$ be fixed and $w_1\in\mathbb{D}$, define the set
	\[	\mathcal{S}_{w_1}=\left\{w_2\in\mathbb{D}:\langle\hat{k}_{w_1},\hat{k}_{w_2}\rangle=q\right\}.\]
		In this regard we have the following result.
	\begin{theorem}\label{imp}
		Let $w_1\in\mathbb{D}$ and $0<q\leq1$. Then the following properties hold.
		\begin{enumerate}
		\item[(i)] $\mathcal{S}_0=\mathbb{T}_{\sqrt{1-q^2}}$ and 
		 $\mathcal{S}_{w_1}=\left\{\lambda^+_{w_1}w_1,\lambda^-_{w_1}w_1\right\}\subseteq\mathbb{D}$, where \[\lambda^{\pm}_{w_1}=\dfrac{|w_1|q^2\pm(1-|w_1|^2)\sqrt{1-q^2}}{|w_1|(1-(1-q^2)|w_1|^2)}\quad \mathrm{for}\ w_1\neq0.\]
		\item[(ii)] $\bigcup\limits_{w_1\in\mathbb{D}}\mathcal{S}_{w_1}=\mathbb{D}$.
	\end{enumerate}
	\end{theorem}
	\begin{proof}
	 (i) Let $w_1=x_1+iy_1$, $w_2=x_2+iy_2\in\mathbb{D}$, where $x_1,x_2,y_1,y_2\in\mathbb{R}$ and $\sqrt{x_1^2+y_1^2}<1,\sqrt{x_2^2+y_2^2}<1$.
	  Now, \begin{eqnarray*}  	
	  		&&\langle\hat{k}_{w_1},\hat{k}_{w_2}\rangle=q \\
	  	&	\implies & \dfrac{\sqrt{(1-x_1^2-y_1^2)(1-x_2^2-y_2^2)}}{1-(x_1-iy_1)(x_2+iy_2)}=q \\
	  	&	\implies & \dfrac{(1-x_1x_2-y_1y_2)\sqrt{(1-x_1^2-y_1^2)(1-x_2^2-y_2^2)}}{(1-x_1x_2-y_1y_2)^2+(x_2y_1-x_1y_2)^2}-i\dfrac{(x_2y_1-x_1y_2)\sqrt{(1-x_1^2-y_1^2)(1-x_2^2-y_2^2)}}{(1-x_1x_2-y_1y_2)^2+(x_2y_1-x_1y_2)^2}=q.
	  \end{eqnarray*}
	  Comparing real and imaginary parts we get,
	  \[
	  	\dfrac{(1-x_1x_2-y_1y_2)\sqrt{(1-x_1^2-y_1^2)(1-x_2^2-y_2^2)}}{(1-x_1x_2-y_1y_2)^2+(x_2y_1-x_1y_2)^2}=q
	  \] and
	  \[
	  	\dfrac{(x_2y_1-x_1y_2)\sqrt{(1-x_1^2-y_1^2)(1-x_2^2-y_2^2)}}{(1-x_1x_2-y_1y_2)^2+(x_2y_1-x_1y_2)^2}=0.
	  \] Since, $\sqrt{(1-x_1^2-y_1^2)(1-x_2^2-y_2^2)}\neq 0$, the following equations follow
	  \begin{eqnarray}\label{imaginary}
	  	x_2y_1-x_1y_2=0,
	  \end{eqnarray} and 
	  \begin{eqnarray}\label{real}
	  	\sqrt{(1-x_1^2-y_1^2)(1-x_2^2-y_2^2)}=q(1-x_1x_2-y_1y_2).
	  \end{eqnarray}
      If $w_1=x_1+iy_1=0$ then $|w_2|^2=1-q^2$ and consequently, $\mathcal{S}_0=\mathbb{T}_{\sqrt{1-q^2}}$. If $w_1\neq0$ then from equation~\eqref{imaginary} we get, $x_2=\lambda x_1$ and $y_2=\lambda y_1$ for some $\lambda\in\mathbb{R}$. 
	   Substituting $(x_2,y_2)=(\lambda x_1,\lambda y_1)$ in equation~\eqref{real} we obtain,
	  \begin{eqnarray*}
	  		&&(1-|w_1|^2)(1-\lambda^2 |w_1|^2)=q^2(1-\lambda |w_1|^2)^2 \\
	  		&\implies &\lambda^2(|w_1|^4-|w_1|^2-q^2|w_1|^4)+2q^2\lambda |w_1|^2+(1-|w_1|^2-q^2)=0 .
	  \end{eqnarray*} The roots of the above quadratic equation are $\dfrac{-q^2|w_1|\pm\sqrt{(1-q^2)(1-|w_1|^2)^2}}{(|w_1|^3-|w_1|-q^2|w_1|^3)}$. It is easy to verify that roots lie in $\left(-\dfrac{1}{|w_1|},\dfrac{1}{|w_1|}\right)$. Thus, $w_2=\lambda^{\pm}_{w_1}w_1$ and consequently, $\mathcal{S}_{w_1}=\left\{\lambda^+_{w_1}w_1,\lambda^-_{w_1}w_1\right\}\subseteq\mathbb{D}$, where $\lambda^{\pm}_{w_1}=\dfrac{|w_1|q^2\pm(1-|w_1|^2)\sqrt{1-q^2}}{|w_1|(1-(1-q^2)|w_1|^2)}$.
	  
	  (ii)  Let $w_1=|w_1|e^{i\theta}$ with $0\leq|w_1|<1$. Now, consider the functions $G_\pm:[0,1)\to\mathbb{R}$ defined by
	  \[
	  	G_\pm(|w_1|)=\lambda^{\pm}_{w_1}|w_1|=\dfrac{|w_1|q^2\pm(1-|w_1|^2)\sqrt{1-q^2}}{1-(1-q^2)|w_1|^2}
	  \] where $0\leq|w_1|<1$ and $0<q\leq1$. Now, we have, $G_-(0)=-\sqrt{1-q^2}$, $G_+(0)=\sqrt{1-q^2}$ and $\lim\limits_{|w_1|\to1^-}G_\pm(|w_1|)=1$. Also, $G_\pm$ are continuous on $[0,1)$. Fix any $r\in[0,1)$.
	  
	 \noindent Case 1: $r\in[\sqrt{1-q^2},1)$.
	  
	  Then $G_+(0)=\sqrt{1-q^2}\leq r<1=\lim\limits_{|w_1|\to1^-}G_+(|w_1|)$. Hence, by the Intermediate Value Theorem there exists $|w_1|\in[0,1)$ such that $G_+(|w_1|)=r$.
	  
	\noindent  Case 2: $r\in[0,\sqrt{1-q^2})$.
	  
	  Then $-\sqrt{1-q^2}<-r<0<1$. Hence, by the Intermediate Value Theorem there exists $|w_1|\in[0,1)$ such that $G_-(|w_1|)=-r$.
	  
	 \noindent Combining two cases we get, for every $r\in[0,1)$, there always exists $|w_1|\in[0,1)$ such that $|G_\pm(|w_1|)|=r$. Let $z=re^{i\phi}\in\mathbb{D}$. We have found $|w_1|\in[0,1)$ and $G_\pm(|w_1|)$ with $|G_\pm(|w_1|)|=r$. Now set $\theta=\phi$ if $G_\pm(|w_1|)=r$ and $\theta=\phi+\pi$ if $G_\pm(|w_1|)=-r$. Thus, $\lambda^{\pm}_{w_1}w_1=(\lambda^{\pm}_{w_1}|w_1|)e^{i\theta}=G_\pm(|w_1|)e^{i\theta}=re^{i\phi}=z$. This completes the proof.
	\end{proof}
	 
	 \begin{remark}
	 	From the relation $\langle\hat{k}_{w_1},\hat{k}_{w_2}\rangle=q$, using equations~\eqref{imaginary} and \eqref{real} we get the following relation between $w_1$ and $w_2$:
	 	\begin{eqnarray}\label{neweq}
	 		\sqrt{(1-|w_1|^2)(1-|w_2|^2)}=q(1-\overline{w_1}w_2).
	 	\end{eqnarray} 
	 \end{remark}
	The definition of the $q$-Berezin range of bounded linear operators on Hardy space is as follows.
	  \begin{definition}
	  	Let $T$ be a bounded linear operator on $H^2(\mathbb{D})$ and $0<q\leq1$.
	  		 The $q$-Berezin range of $T$ is defined by
	  		\begin{eqnarray*}
	  			\mathrm{Ber}_q(T)
	  		 := 
	  		 \left\{\langle T\hat{k}_{w_1},\hat{k}_{ w_2}\rangle:w_1\in\mathbb{D},w_2\in\mathcal{S}_{w_1}\right\}.
	 \end{eqnarray*}
	  \end{definition}
	  	 Ber$_q(T)$ can also be alternatively expressed as
	 
	 	\[
	  	\mathrm{Ber}_q(T):=\bigcup\limits_{w_1\in\mathbb{D}}\mathcal{S}^{(T)}_{w_1},\quad \mathrm{where}\quad \mathcal{S}^{(T)}_{w_1}=\left\{\langle T\hat{k}_{w_1},\hat{k}_{w_2}\rangle:w_2\in \mathcal{S}_{w_1}\right\}.\]
	  Let $w_1=x_1+iy_1$, $w_2=x_2+iy_2$. Then from equations~\eqref{imaginary} and \eqref{real} we have, $w_2=\lambda^\pm_{w_1}w_1$ where $\lambda^{\pm}_{w_1}=\dfrac{|w_1|q^2\pm(1-|w_1|^2)\sqrt{1-q^2}}{|w_1|(1-(1-q^2)|w_1|^2)}$ for $0<|w_1|<1$. Hence, $\overline{w_1}w_2=\lambda^{\pm}_{w_1}|w_1|^2$. Also, $w_1=0$ implies $\overline{w_1}w_2=0$.
	 
	  The next lemma is crucial for the remaining part of the paper.
	  	\begin{lemma}\label{range}
	  		 For $0<q\leq1$	we have,
	  		\[\left\{\overline{w_1}w_2:w_1\in\mathbb{D}\setminus\{0\},w_2\in\mathcal{S}_{w_1}\right\}=
	  		\left\{\lambda^{\pm}_{w_1}|w_1|^2:0<|w_1|<1\right\}=\left[\dfrac{q-1}{q+1},1\right).
	  	\] 
	  	\end{lemma}
	  	\begin{proof} If $q=1$ then clearly $\lambda^\pm_{w_1}=1$ and so $	\left\{|w_1|^2:0<|w_1|<1\right\}=\left(0,1\right)$.
	  		  For $0<|w_1|<1$ and $0<q<1$ it is obvious that $\lambda^-_{w_1}|w_1|^2\leq\lambda^+_{w_1}|w_1|^2$. Let $|w_1|=r$ and $k=\sqrt{1-q^2}$. Consider the function $\phi:(0,1)\to\mathbb{R}$ defined by
	  	\begin{eqnarray*}
	  		\phi(r)=\lambda^-_{w_1}r^2=\dfrac{r^2(1-k^2)-rk(1-r^2)}{1-k^2r^2}
            = \dfrac{r(r-k)}{1-kr}.
	  	\end{eqnarray*} Now,
	  	\begin{eqnarray*} 
	  			&&\phi'(r)=
	  			 \dfrac{2r-k-kr^2}{(1-kr)^2}=0 \\
	  			&\implies & r=\dfrac{1\pm\sqrt{1-k^2}}{k}=\sqrt{\dfrac{1+q}{1-q}},\sqrt{\dfrac{1-q}{1+q}}.
	  	\end{eqnarray*} Since, $0<r<1$ we have, $r=\sqrt{\dfrac{1-q}{1+q}}$, and $\phi''\left(\sqrt{\dfrac{1-q}{1+q}}\right)=\dfrac{1}{q}>0$ for $0<q<1$. Hence, 
	  	\[
	  			\phi\left(\sqrt{\dfrac{1-q}{1+q}}\right)=\dfrac{q^2\dfrac{1-q}{1+q}-\sqrt{\dfrac{1-q}{1+q}}\left(1-\dfrac{1-q}{1+q}\right)\sqrt{1-q^2}}{1-\dfrac{1-q}{1+q}(1-q^2)}
	  			= \dfrac{q-1}{q+1}.
	  	\]Also, for $0<r<1$, $0<q<1$ we have $\lambda^+_{w_1}r^2>0$. Consider the function $\psi:(0,1)\to\mathbb{R}$ defined by 
	  	\[
	  		\psi(r)=\lambda^+_{w_1}r^2=\dfrac{r^2(1-k^2)+rk(1-r^2)}{1-k^2r^2}=\dfrac{r(r+k)}{1+kr}. \]
	  		We have,\[ \psi'(r)=\dfrac{2r+k+kr^2}{(1+kr)^2} >0.
	  	\] This implies $\psi(r)$ is an increasing function of $r$. Also, $\lim\limits_{r\to1^-}\psi(r)=1$. 
	  	Thus, \[\{\lambda^{\pm}_{w_1}|w_1|^2:0<|w_1|<1\}=\left[\dfrac{q-1}{q+1},1\right).\]
	  \end{proof}

     \section{Convexity of $q$-Barezin range of some classes of operators}\label{sec3}
In this section, we investigate the criteria that ensure the convexity of the $q$-Berezin range for several important classes of operators, including diagonal, Toeplitz, weighted shift, and composition operators. We begin with an important observation in the finite-dimensional setting, which highlights certain structural limitations of the $q$-Berezin range.

Consider $\mathbb{C}^n$ as a reproducing kernel Hilbert space on the finite set $\Omega = \{1,2,\dots,n\}$, where each vector $v = (v_1, v_2, \dots, v_n) \in \mathbb{C}^n$ is identified with a function on $\Omega$ via $v(j) = v_j$. Let $\{e_j\}_{j=1}^n$ denote the standard orthonormal basis of $\mathbb{C}^n$, so that
\[
e_j(i) =
\begin{cases}
1, & \text{if } i = j, \\
0, & \text{if } i \neq j.
\end{cases}
\]
Then $\mathbb{C}^n$ is a reproducing kernel Hilbert space with reproducing kernel given by
\[
k_j(i) = \langle e_j, e_i \rangle_{\mathbb{C}^n}, \quad i,j \in \Omega.
\]
In particular, each kernel function is normalized, that is, $k_j = \hat{k}_j$ for all $j = 1,2,\dots,n$.

In order to determine the $q$-Berezin range of $A = (a_{ij}) \in M_n(\mathbb{C})$, one requires the normalized reproducing kernels to satisfy
\[
\langle \hat{k}_j, \hat{k}_i \rangle_{\mathbb{C}^n} = q, \quad \text{for } 0 < q \leq 1.
\]
However, in this setting, we have
\[
\langle \hat{k}_j, \hat{k}_i \rangle_{\mathbb{C}^n} \in \{0,1\},
\]
which immediately implies the following

\[
\mathrm{Ber}_q(A) =
\begin{cases}
\emptyset, & \text{if } 0<q<1\\
\mathrm{diag}(A), & \text{if } q=1,\\
\{a_{ij} : i \neq j\}, & \text{if } q=0.
\end{cases}
\]
Consequently, $\mathrm{Ber}_1(A)$ is convex if and only if $A$ has constant diagonal entries, and $\mathrm{Ber}_0(A)$ is convex if and only if $a_{ij}$'s are constant for $i \neq j.$

This observation motivates the study of infinite-dimensional settings, where richer geometric behavior of the $q$-Berezin range may be observed. We start with diagonal operators.

\subsection{Diagonal operators}
	Let $T \in \mathscr{B}(H^2(\mathbb{D}))$ be a diagonal operator. Then with respect to the basis $\{e_n\}$, $e_n(z)=z^n$, $T$ can be represented by
\[
T(f)=\sum_{n=0}^{\infty}\alpha_n a_n z^n
=\sum_{n=0}^{\infty}\alpha_n \langle f, z^n\rangle z^n,
\]
for $f(z)=\sum\limits_{n=0}^{\infty} a_n z^n$, where $\{\alpha_n\}_{n=0}^{\infty}$ is a bounded sequence in $\mathbb{C}$.
\begin{theorem}\label{third}
Let $T(f)=\sum\limits_{n=0}^{\infty}\alpha_n\langle f,z^n\rangle z^n,f\in H^2(\mathbb{D})$, where $\{\alpha_n\}_{n=0}^\infty$ is a bounded sequence in $\mathbb{C}$ and $0<q\leq1$. Then the following statements are equivalent:
\begin{enumerate}
     \item[(i)] The set $\mathrm{Ber}_q(T)$ is convex.
      \item[(ii)] The set $\mathrm{Ber}_q(T)$ is a line-segment in $\mathbb{C}$.
    \item[(iii)] All $\alpha_n$, $n\geq0$ are lying on a straight line in $\mathbb{C}$.
\end{enumerate}
\end{theorem}
\begin{proof}
     	Let $w_1\in\mathbb{D}$ and $w_2\in \mathcal{S}_{w_1}$, that is, $\langle\hat{k}_{w_1},\hat{k}_{w_2}\rangle=q$. Then
     	\begin{eqnarray}
     		\langle T\hat{k}_{w_1},\hat{k}_{w_2}\rangle 
     		&=&	\sqrt{(1-|w_1|^2)(1-|w_2|^2)}\left\langle\sum_{n=0}^{\infty}\alpha_n\langle k_{w_1},z^n\rangle z^n,k_{w_2}\right\rangle\nonumber\\
     		& =& q(1-\overline{w_1}w_2)\sum_{n=0}^{\infty}\alpha_n(\overline{w_1}w_2)^n\quad  \text{(using equation~\eqref{neweq})}\nonumber.
     	\end{eqnarray} In  particular, if $w_1=0$ then $\langle T\hat{k}_{w_1},\hat{k}_{w_2}\rangle =q\alpha_0$. Now, from Lemma~\ref{range} we have,
     	\begin{eqnarray}\label{finite}
     \  \mathrm{Ber}_q(T)=\left\{q(1-t)\sum_{n=0}^\infty\alpha_n t^n:t\in\left[\dfrac{q-1}{q+1},1\right)\right\}.
     	\end{eqnarray}
        Since, Ber$_q(T)$ is the image of $\left[\dfrac{q-1}{q+1},1\right)$ under a continuous function, it is convex if and only if it is a line-segment in $\mathbb{C}$. This proves (i)$\iff$(ii).
        
For (ii)$\iff$(iii), let $F(t)=q(1-t)\sum\limits_{n=0}^\infty\alpha_nt^n$ and $\alpha_n=a_n+ib_n$ for all $n\geq0$. Then \begin{eqnarray*}
 F(t)&=&q(1-t)\sum\limits_{n=0}^\infty a_nt^n+iq(1-t)\sum\limits_{n=0}^\infty b_nt^n\\
 &=&g(t)+ih(t),
 \end{eqnarray*}where $g(t)=q(1-t)\sum\limits_{n=0}^\infty a_nt^n$ and $h(t)=q(1-t)\sum\limits_{n=0}^\infty b_nt^n$.
This represents a line-segment in $\mathbb{C}$ if and only if $Ag(t)+Bh(t)=C$ holds for all $t\in\left[\dfrac{q-1}{q+1},1\right)$ and for some $A,B,C\in\mathbb{R}$ with $(A,B)\neq(0,0)$. Now, 
\begin{eqnarray*}
    &&Ag(t)+Bh(t)=C\quad \mathrm{for\ all}\ t\in \left[\dfrac{q-1}{q+1},1\right)\\
    &\iff& q(1-t)\left(\sum_{n=0}^\infty(Aa_n)t^n+\sum_{n=0}^\infty(Bb_n)t^n\right)=C\quad \mathrm{for\ all}\ t\in \left[\dfrac{q-1}{q+1},1\right)\\
    &\iff& \sum_{n=0}^\infty(Aa_n+Bb_n)t^n=\dfrac{C}{q}\sum_{n=0}^\infty t^n \quad \mathrm{for\ all}\ t\in \left[\dfrac{q-1}{q+1},1\right)\\
    &\iff& Aa_n+Bb_n=\dfrac{C}{q},\quad\ \mathrm{for\ all}\ n\geq0,
\end{eqnarray*} where the final equivalence follows from the Uniqueness Theorem for Power Series. Hence, $\mathrm{Ber}_q(T)$ is a line-segment in $\mathbb{C}$ if and only if all $\alpha_n$, $n\geq0$ are lying on a straight line in $\mathbb{C}$.
\end{proof}
\begin{corollary}
Let $T(f)=\sum\limits_{n=0}^{\infty}\alpha_n\langle f,z^n\rangle z^n,f\in H^2(\mathbb{D})$, where $\{\alpha_n\}_{n=0}^\infty$ is a bounded sequence in $\mathbb{C}$ and $0<q\leq1$. Then
    \begin{enumerate}
        \item [(i)] Ber$_q(T)$ is a singleton if and only if $\alpha_n=\alpha_0$, for all $n\geq0$ and for some $\alpha_0\in\mathbb{C}$,
        \item [(ii)] $\Im\left\{\mathrm{Ber}_q(T)\right\}=\{0\}$ if and only if $\Im\{\alpha_n\}=0$.
    \end{enumerate}
\end{corollary}
\begin{proof}
    (i) If $\alpha_n=\alpha_0$, for all $n\geq0$, then from equation~\eqref{finite} we have, Ber$_q(T)=\{q\alpha_0\}$. Conversely, let Ber$_q(T)$ is a singleton. As, $q(1-t)\sum\limits_{n=0}^\infty\alpha_n t^n=q\alpha_0$ at $t=0$ then $\mathrm{Ber}_q(T)=\{q\alpha_0\}$, for all $t\in\left[\dfrac{q-1}{q+1},1\right).$ This implies \[\sum\limits_{n=0}^\infty\alpha_n t^n=\alpha_0\sum\limits_{n=0}^\infty t^n,\quad\ \mbox{for all}\ t\in\left[\dfrac{q-1}{q+1},1\right).\] By the Uniqueness Theorem for Power Series it follows that $\alpha_n=\alpha_0$ for all $n\geq0$.

    (ii) Let $\Im\left\{\mathrm{Ber}_q(T)\right\}=0$. Then $q(1-t)\sum\limits_{n=0}^\infty\alpha_n t^n\in\mathbb{R}$ for all $t\in\left[\dfrac{q-1}{q+1},1\right)$. Considering $\alpha_n=a_n+ib_n$, we must have, $\sum\limits_{n=0}^\infty b_n t^n=0$ for all $t\in\left[\dfrac{q-1}{q+1},1\right)$. This follows that $b_n=0$ for all $n\geq0$ by the Uniqueness Theorem for Power Series. The converse part is obvious.
\end{proof}
      \begin{example}
           Let $\alpha_n= 1+i^n$. Then we have, \[\mathrm{Ber}_q(T)=\left\{q\left(1+\dfrac{1-t}{1-it}\right):t\in\left[\dfrac{q-1}{q+1},1\right)\right\},\] which is not convex as all $\alpha_n$ are not lying on a straight line in $\mathbb{C}$. This also can be verified from the following figure for several values of $q$.
     		\begin{figure}[H]
     			\begin{center}
     			\includegraphics[width=0.50\columnwidth]{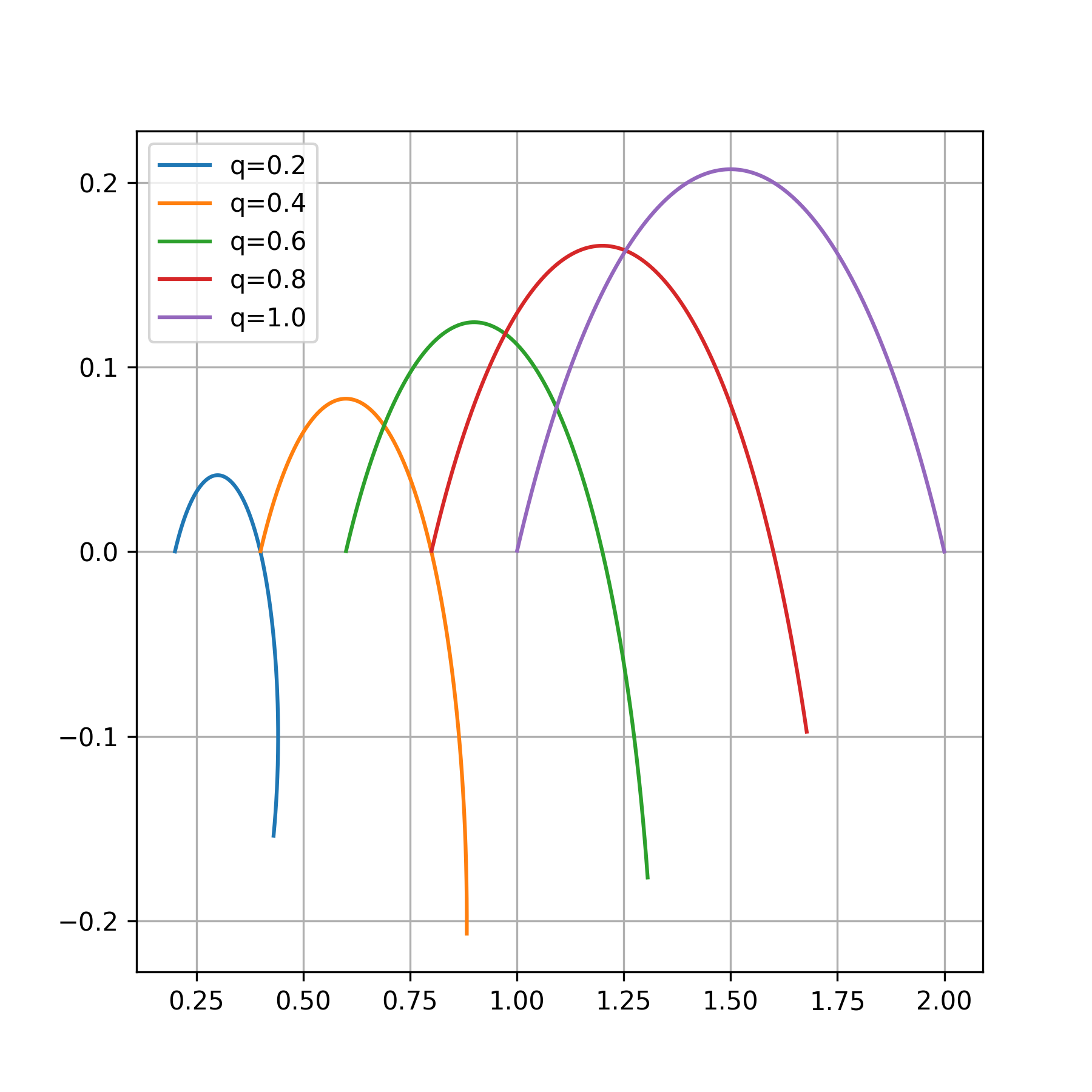}
     			\caption{Ber$_q(T)$ on $H^2(\mathbb{D})$ for $q=0.2,0.4,0.6,0.8,1.0$ and $\alpha_n=1+i^n$.}
                \label{FIG:1}
                \end{center}
     		\end{figure}
     \end{example}
     The $q$-Berezin range of diagonal operators helps us to establish an important observation, that the $q$-Berezin range does not enjoy the unitary invariance property, unlike the $q$-numerical range.
     	\begin{corollary}\label{diagcoro}
        $q$-Berezin range is not unitarily invariant.
     \end{corollary}
     \begin{proof}
      Let $T_1=\langle f,z\rangle z$ and $T_2=\langle f,z^2\rangle z^2$ be rank one operators on $H^2(\mathbb{D})$, and define a unitary operator $U$ by $U(z)=z^2$, $U(z^2)=z$, and $U(z^n)=z^n$ for $n\neq1,2$. Then $T_2=UT_1U^*$. Now, we need to find Ber$_q(T_1)$ and Ber$_q(T_2)$. To this end, let $T(f)=\langle f,z^k\rangle z^k$, $k\in\mathbb{N}$. Then from equation~\eqref{finite} 	\[
     	\mathrm{Ber}_q(T)=\left\{q(1-t)t^k:t\in\left[\dfrac{q-1}{q+1},1\right)\right\}.
     	\] For every $k\in\mathbb{N}$, 
$\displaystyle \max \left\{q(1-t)t^k:t\in\left[\dfrac{q-1}{q+1},1\right)\right\}
= \frac{q}{k+1}\left(\frac{k}{k+1}\right)^k$, and
\[\displaystyle \min \left\{q(1-t)t^k:t\in\left[\dfrac{q-1}{q+1},1\right)\right\}
=
\begin{cases}
\dfrac{2q}{q+1}\left(\dfrac{q-1}{q+1}\right)^k, & k\ \text{odd},\\
0, & k\ \text{even.}
\end{cases}\] 
Hence, \[
\mathrm{Ber}_q(T)=
\begin{cases}
\left[\dfrac{2q}{q+1}\left(\dfrac{q-1}{q+1}\right)^k,\; 
\dfrac{q}{k+1}\left(\dfrac{k}{k+1}\right)^k\right], 
& \text{for odd } k, \\
\left[0,\; \dfrac{q}{k+1}\left(\dfrac{k}{k+1}\right)^k\right], 
& \text{for even } k.
\end{cases}
\] Thus, Ber$_q(T_1)=\left[\dfrac{2q(q-1)}{(q+1)^2},\dfrac{q}{4}\right]$ and Ber$_q(T_2)=\left[0,\dfrac{4q}{27}\right]$, which are different for $0<q\leq1$. 
     \end{proof}

\subsection{Toeplitz operators}
 Let $L^2(\mathbb{T})$ denote the Hilbert space of square-integrable functions on the unit circle $\mathbb{T}$. The Hardy space $\widetilde{H}^2$ is defined as the closed subspace
\[
\widetilde{H}^2
=
\left\{
f\in L^2(\mathbb{T}) :
f(e^{i\theta})=\sum_{n=0}^{\infty} a_n e^{in\theta},
\quad \sum_{n=0}^{\infty}|a_n|^2<\infty
\right\}.
\]

The space $\widetilde{H}^2$ can be naturally identified with the Hardy space $H^2(\mathbb{D})$ of analytic functions on the unit disc $\mathbb{D}$. Indeed, for each
$f(e^{i\theta})=\sum\limits_{n=0}^{\infty} a_n e^{in\theta}\in \widetilde{H}^2,$ we associate the analytic function; $
\widetilde{f}(z)=\sum\limits_{n=0}^{\infty} a_n z^n$, $z\in\mathbb{D}.$ This correspondence defines an isometric isomorphism between $\widetilde{H}^2$ and $H^2(\mathbb{D})$.

For $0<r<1$, define
\[
f_r(e^{i\theta})=\widetilde{f}(re^{i\theta})
=
\sum_{n=0}^{\infty} a_n r^n e^{in\theta}.
\]
Then $f_r\in \widetilde{H}^2$ for every $0<r<1$, and the following approximation result holds.

\begin{lemma}\cite[Theorem~1.1.10]{martinez2007introduction}
Let $f\in \widetilde{H}^2$, and let $f_r$ be defined as above. Then
\[
\lim_{r\to1^-}\|f-f_r\|_{L^2(\mathbb{T})}=0.
\]
\end{lemma}

The following result justifies the identification of $H^2(\mathbb{D})$ with $\widetilde{H}^2$ via radial boundary values.

\begin{lemma}\cite[Corollary~1.1.28]{martinez2007introduction}
If $\widetilde{f}\in H^2(\mathbb{D})$, then
\[
\lim_{r\to1^-}\widetilde{f}(re^{i\theta})=f(e^{i\theta})
\]
for almost every $e^{i\theta}\in\mathbb{T}$, where $f\in\widetilde{H}^2$ is the boundary function associated with $\widetilde{f}$.
\end{lemma}

Moreover, the values of $\widetilde{f}$ in the unit disc are recovered from its boundary values through the Poisson integral formula:
\[
\widetilde{f}(re^{i\theta})
=
\frac{1}{2\pi}
\int_{0}^{2\pi}
\frac{1-r^2}{1-2r\cos(\theta-t)+r^2}
\,f(e^{it})\,dt,
\quad re^{i\theta}\in\mathbb{D}.
\]

Let $P_+$ denote the orthogonal projection from $L^2(\mathbb{T})$ onto $\widetilde{H}^2$, and set $P_-=I-P_+$. For $\phi\in L^\infty(\mathbb{T})$, the Toeplitz operator $T_\phi$ with symbol $\phi$ is defined by
\[
T_\phi f=P_+(\phi f),
\qquad f\in \widetilde{H}^2.
\]
Then $T_\phi$ is a bounded linear operator on $\widetilde{H}^2$, satisfying
\[
\|T_\phi\|\le \|\phi\|_\infty.
\]
Via the identification $\widetilde{H}^2\cong H^2(\mathbb{D})$, we regard $T_\phi$ as a bounded linear operator on $H^2(\mathbb{D})$.

\begin{proposition}~\label{new prop} 
    Let $\phi\in L^\infty(\mathbb{T})$. Then the $q$-Berezin range of the Toeplitz operator $T_\phi$ is given by
    \[\mathrm{Ber}_q(T_{\phi})=\left\{\dfrac{q}{2\pi}\int_0^{2\pi}\displaystyle\dfrac{\phi(e^{it)}(1-\overline{w_1}w_2)}{(1-\overline{w_1}e^{it})(1-w_2e^{-it})}\, dt:w_1\in\mathbb{D},w_2\in\mathcal{S}_{w_1}\right\}.\]
\end{proposition}
\begin{proof}
    Let $w_1\in\mathbb{D}$ and $w_2\in\mathcal{S}_{w_1}$, that is, $\langle \hat{k}_{w_1},\hat{k}_{w_2}\rangle=q$. Then
    \begin{eqnarray*}
       \langle T_{\phi}\hat{k}_{w_1},\hat{k}_{w_2}\rangle &=&\langle \phi\hat{k}_{w_1},P_+\hat{k}_{w_2}\rangle\\
       &=&\langle \phi\hat{k}_{w_1},\hat{k}_{w_2}\rangle\\
       &=&\dfrac{\sqrt{(1-|w_1|^2)(1-|w_2|^2)}}{2\pi}\int_0^{2\pi}\displaystyle\dfrac{\phi(e^{it})}{(1-\overline{w_1}e^{it})(1-w_2e^{-it})}\, dt\\
       &=&\dfrac{q}{2\pi}\int_0^{2\pi}\displaystyle\dfrac{\phi(e^{it)}(1-\overline{w_1}w_2)}{(1-\overline{w_1}e^{it})(1-w_2e^{-it})}\, dt \quad\text{(using equation~\eqref{neweq})}.
    \end{eqnarray*}
    Hence the result follows.
\end{proof}

 In \cite{englivs1995toeplltz}, it is established that the Berezin transform of Toeplitz operators over Hardy spaces is the Poisson integral of its symbol. For $q=1$ this result follows form Proposition \ref{new prop}.
 \begin{corollary}
  Let $\phi\in L^\infty(\mathbb{T})$. Then 
  \[\mathrm{Ber}(T_\phi)=\left\{\dfrac{1}{2\pi}\displaystyle\int_0^{2\pi}\dfrac{\phi(e^{it)}(1-|w_1|^2)}{|1-\overline{w_1}e^{it}|^2}\, dt:w_1\in\mathbb{D}\right\}.\] 
  \end{corollary}
  \begin{proof}
Considering $q=1$ Theorem~\ref{imp} yields $w_1=0$ implies $w_2=0$ and also for $w_1\neq0$, we have $\lambda_{w_1}^\pm=1$ and hence, $w_1=w_2$. Then the integral $\widetilde\phi(w_1,w_2)$ reduces to 
  \[\widetilde\phi(w_1):=\dfrac{1}{2\pi}\displaystyle\int_0^{2\pi}\dfrac{\phi(e^{it)}(1-|w_1|^2)}{|1-\overline{w_1}e^{it}|^2}\, dt,\] which is the harmonic extension of $\phi$ into $\mathbb{D}$ such that the Poisson kernel, $P(w_1,e^{it})=\dfrac{1-|w_1|^2}{|1-\overline{w_1}e^{it}|^2}$. Thus, 
  \[\mathrm{Ber}(T_\phi)=\{\widetilde\phi(w_1):w_1\in\mathbb{D}\}\quad .\] 
  
 \end{proof}
 \begin{example}
     Consider a Toeplitz operator with the symbol $\phi(e^{it})=2e^{it}$ then from Proposition~\ref{new prop} we have,
     \[\mathrm{Ber}_q(T_{\phi})=\left\{\dfrac{q}{2\pi}\int_0^{2\pi}\displaystyle\dfrac{2e^{it}(1-\overline{w_1}w_2)}{(1-\overline{w_1}e^{it})(1-w_2e^{-it})}\, dt:w_1\in\mathbb{D},w_2\in\mathcal{S}_{w_1}\right\}.\]
     Let $z=e^{it}$ then $dt=\dfrac{dz}{iz}$ and $\dfrac{1}{z}=e^{-it}$. Then by a simple calculation we obtain,
     \begin{eqnarray*}
         \dfrac{q}{2\pi}\int_0^{2\pi}\displaystyle\dfrac{2e^{it}(1-\overline{w_1}w_2)}{(1-\overline{w_1}e^{it})(1-w_2e^{-it})}\, dt= \dfrac{q}{2\pi i}\oint_{|z|=1}\dfrac{2z(1-\overline{w_1}w_2)}{(1-\overline{w_1}z)(z-w_2)}\, dz.
     \end{eqnarray*}  Now, the integrand has pole only at $z=w_2$ inside the unit circle. Hence, by the Residue Theorem, we get, 
     \[\dfrac{q}{2\pi i}\oint_{|z|=1}\dfrac{2z(1-\overline{w_1}w_2)}{(1-\overline{w_1}z)(z-w_2)}\, dz=\dfrac{2qw_2(1-\overline{w_1}w_2)}{(1-\overline{w_1}w_2)}=2qw_2.\]
     Thus, \[\mathrm{Ber}_q(T_\phi)=\bigcup_{w_1\in\mathbb{D}}\{2qw_2:w_2\in\mathcal{S}_{w_1}\}=2q\mathbb{D}.\] Hence, the $q$-Berezin range of $T_\phi$ is an open disc centered at origin with radius $2q$.
 \end{example}
 
 With reference to Proposition~\ref{new prop} we can define $q$-Poisson kernel, the $q$-analouge of Poisson kernel, as
 \[P(w_1,w_2,e^{it}):=\dfrac{(1-\overline{w_1}w_2)}{(1-\overline{w_1}e^{it})(1-w_2e^{-it})}.\] 
 It also satisfies the relation
 \[\dfrac{1}{2\pi}\int_0^{2\pi}\displaystyle P(w_1,w_2,e^{it}) = 1.\]
 Also, in a similar manner, the $q$-Poisson transform of $\phi \in L^\infty(\mathbb{T})$ is defined as \[\widetilde\phi(w_1,w_2):=\dfrac{1}{2\pi}\displaystyle\int_0^{2\pi}\dfrac{\phi(e^{it)}(1-\overline{w_1}w_2)}{(1-\overline{w_1}e^{it})(1-w_2e^{-it})}\, dt\]
 where $\langle \hat{k}_{w_1},\hat{k}_{w_2}\rangle=q.$ It is straightforward to verify that when $q=1$, the $q$-Poisson kernel and the $q$-Poisson transform coincide with the classical Poisson kernel and Poisson transform respectively.

\begin{remark} 
From Proposition \ref{new prop}, it follows that the $q$-Berezin range of $T_\phi$ is the range of the $q$-multiple of the $q$-Poisson transform $\widetilde\phi(w_1,w_2)$ of the symbol $\phi$ and the $q$-berezin number is given by

\[ \mbox{ber}_q(T_\phi)=\sup\limits_{w_1\in\mathbb{D},w_2\in\mathcal{S}_{w_1}}|q\widetilde\phi(w_1,w_2)|.\]
\end{remark}

Let $\phi\in L^\infty(\mathbb{T})$ be harmonic in $\mathbb{D}$ then we denote $f=P_+\phi$ and $g=P_-\phi$ such that $f,\overline{g}\in H^2(\mathbb{D})$. Hence, we have, $\phi=f+g$ and $T_\phi=T_f+T_g$. 
\begin{corollary}\label{new thm}
    Let $T_\phi$ be the Toeplitz operator with symbol $\phi\in L^\infty(\mathbb{T})$ and $\phi$ be harmonic in $\mathbb{D}$. Then the $q$-Berezin range of $T_\phi$ is given by
    \[\mathrm{Ber}_q(T_\phi)=\left\{q(f(w_2)+g(w_1)):w_1\in\mathbb{D},w_2\in\mathcal{S}_{w_1}\right\}.\]
\end{corollary}
\begin{proof}
Let $w_1\in \mathbb{D}$ and $w_2\in \mathcal{S}_{w_1}$, that is, $\langle \hat{k}_{w_1},\hat{k}_{w_2}\rangle=q$. Then for $f\in H^2(\mathbb{D})$, we have, 
      	\[
      	  \langle T_{f}\hat{k}_{w_1},\hat{k}_{w_2}\rangle 
        = \dfrac{1}{\|k_{w_1}\|\|k_{w_2}\|}\langle fk_{w_1},k_{w_2}\rangle 
        = \dfrac{1}{\|k_{w_1}\|\|k_{w_2}\|}f(w_2)\langle k_{w_1},k_{w_2}\rangle 
      	= qf(w_2). 
\] Also, for $\overline{g}\in H^2(\mathbb{D})$, we get, 
\[\langle T_{g}\hat{k}_{w_1},\hat{k}_{w_2}\rangle 
        = \dfrac{1}{\|k_{w_1}\|\|k_{w_2}\|}\langle k_{w_1},\overline{g}k_{w_2}\rangle 
        = \dfrac{1}{\|k_{w_1}\|\|k_{w_2}\|}g(w_1)\langle k_{w_1},k_{w_2}\rangle 
      	= qg(w_1). \] 
Therefore, the result follows.
\end{proof}
\begin{remark}
    The $q$-Berezin number, ber$_q(T_\phi)=\sup\limits_{w_1\in\mathbb{D},w_2\in\mathcal{S}_{w_1}}|q(f(w_2)+g(w_1))|$.
\end{remark}


A Toeplitz operator $T_\phi$ is said to be \emph{analytic} if its symbol $\phi$ belongs to $H^\infty(\mathbb{D})$, the algebra of bounded analytic functions on $\mathbb{D}$. In this case, multiplication by $\phi$ leaves $H^2(\mathbb{D})$ invariant, and hence
\[
T_\phi f = \phi f, \qquad f \in H^2(\mathbb{D}).
\]
Accordingly, we denote $T_\phi$ by $M_\phi$ whenever $\phi \in H^\infty(\mathbb{D})$, emphasizing that $M_\phi$ acts as the multiplication operator
\[
(M_\phi f)(z) = \phi(z) f(z), \qquad z \in \mathbb{D}.
\]

A Toeplitz operator is called \emph{coanalytic} if its symbol is of the form $\phi = \overline{\psi}$ for some $\psi \in H^\infty(\mathbb{D})$. Equivalently, $T_\phi^* = T_\psi$ is an analytic Toeplitz operator.

An immediate consequence is the following corollary.
\begin{corollary}
Ber$_q(M_\phi)$ is convex if and only if $\phi(\mathbb{D})$ is convex.
\end{corollary}
\begin{proof}
 From the proof of Corollary~\ref{new thm} we have, 
\begin{eqnarray}\label{multeq}
    \mathrm{Ber}_q(M_\phi)
      		=\bigcup\limits_{w_1\in\mathbb{D}}\left\{q\phi(w_2):w_2\in \mathcal{S}_{w_1}\right\}
      		=q\phi(\mathbb{D}).
      	\end{eqnarray}
        This completes the proof.
      \end{proof}
      \begin{remark}
It is known that for $q=1$, the Berezin range $\mathrm{Ber}(M_\phi)$ is convex if and only if $\phi(\mathbb{D})$ is convex (see \cite[Proposition~3.2]{cowen2022convexity}).
\end{remark}
In the following examples we characterize the $q$-Berezin range of particular classes of Toeplitz operators.
      \begin{example}\label{shift}
      Let $T_\phi$ be the Toeplitz operator with harmonic symbol $\phi(z)=z+\overline{z},z\in\mathbb{D}$. Then $T_\phi$ becomes a constant Jacobi matrix. From Corollary~\ref{new thm} we obtain,  \[ \mathrm{Ber}_q(T_\phi)=\left\{q(w_2+\overline{w_1}):w_1\in\mathbb{D},w_2\in\mathcal{S}_{w_1}\right\}.\]  
      		
      		\noindent If $w_1=0$ then $\langle T_\phi\hat{k}_{w_1},\hat{k}_{w_2}\rangle =qw_2$, where $w_2\in\mathbb{T}_{\sqrt{1-q^2}}$. If $w_1\neq0$ then for $\theta\in[0,2\pi)$,
      		\begin{eqnarray*}
      			\langle T_\phi\hat{k}_{w_1},\hat{k}_{w_2}\rangle 
=q(\lambda^\pm_{w_1}|w_1|+|w_1|)\cos\theta+iq(\lambda^\pm_{w_1}|w_1|-|w_1|)\sin\theta.
      		\end{eqnarray*} Combining the above cases we get, \[\mathrm{Ber}_q(T_\phi)=\Gamma_+\cup \Gamma_-\cup q\mathbb{T}_{\sqrt{1-q^2}},\]
where
\[
\Gamma_+=\left\{ q\Big(|w_1|\cos\theta(\lambda^+_{w_1}+1)
+ i|w_1|\sin\theta(\lambda^+_{w_1}-1)\Big) : 0<|w_1|<1,\; 0\leq\theta<2\pi \right\},
\] and
\[
\Gamma_-=\left\{ q\Big(|w_1|\cos\theta(\lambda^-_{w_1}+1)
+ i|w_1|\sin\theta(\lambda^-_{w_1}-1)\Big) : 0<|w_1|<1,\; 0\leq\theta<2\pi \right\}.
\]
Thus, Ber$_q(T_\phi)$ corresponding to the symbol 
$\phi(z)=z+\overline{z}$ for $z\in\mathbb{D}$, can be described as the union of two families of ellipses together with the circle centered at the origin with radius $q\sqrt{1-q^2}$. For $q=0.4$, we illustrate the sets $\Gamma_+$, $\Gamma_-$, and 
$q\mathbb{T}_{\sqrt{1-q^2}}$, as well as their union, in the figures below. From these plots, one can observe that $\mathrm{Ber}_q(T_\phi)$ appears to be non-convex.
\end{example}
      		\begin{figure}[H]
      			\centering
      			\includegraphics[width=0.32\columnwidth]{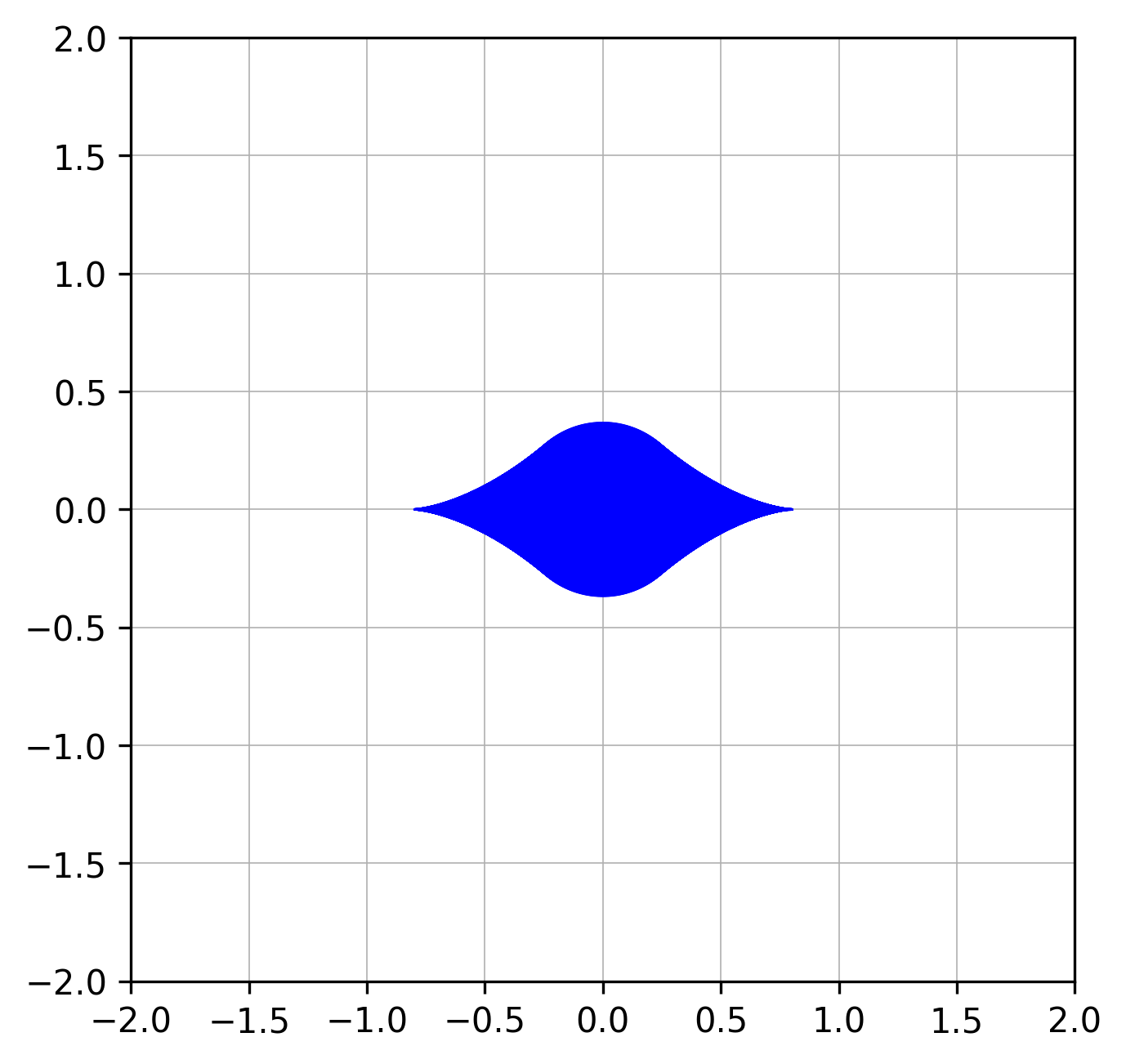}
      			\includegraphics[width=0.32\columnwidth]{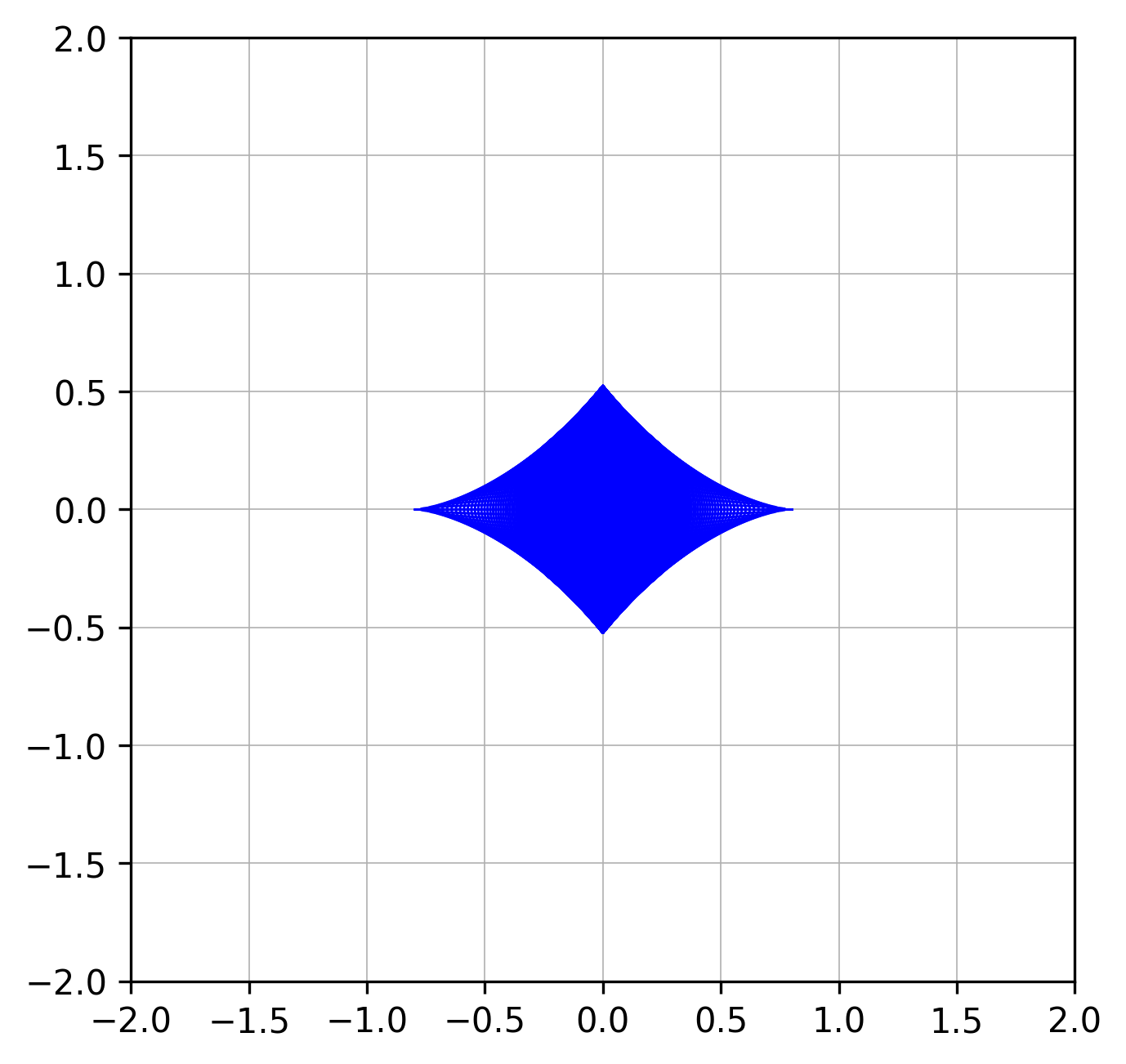}
      			\includegraphics[width=0.32\columnwidth]{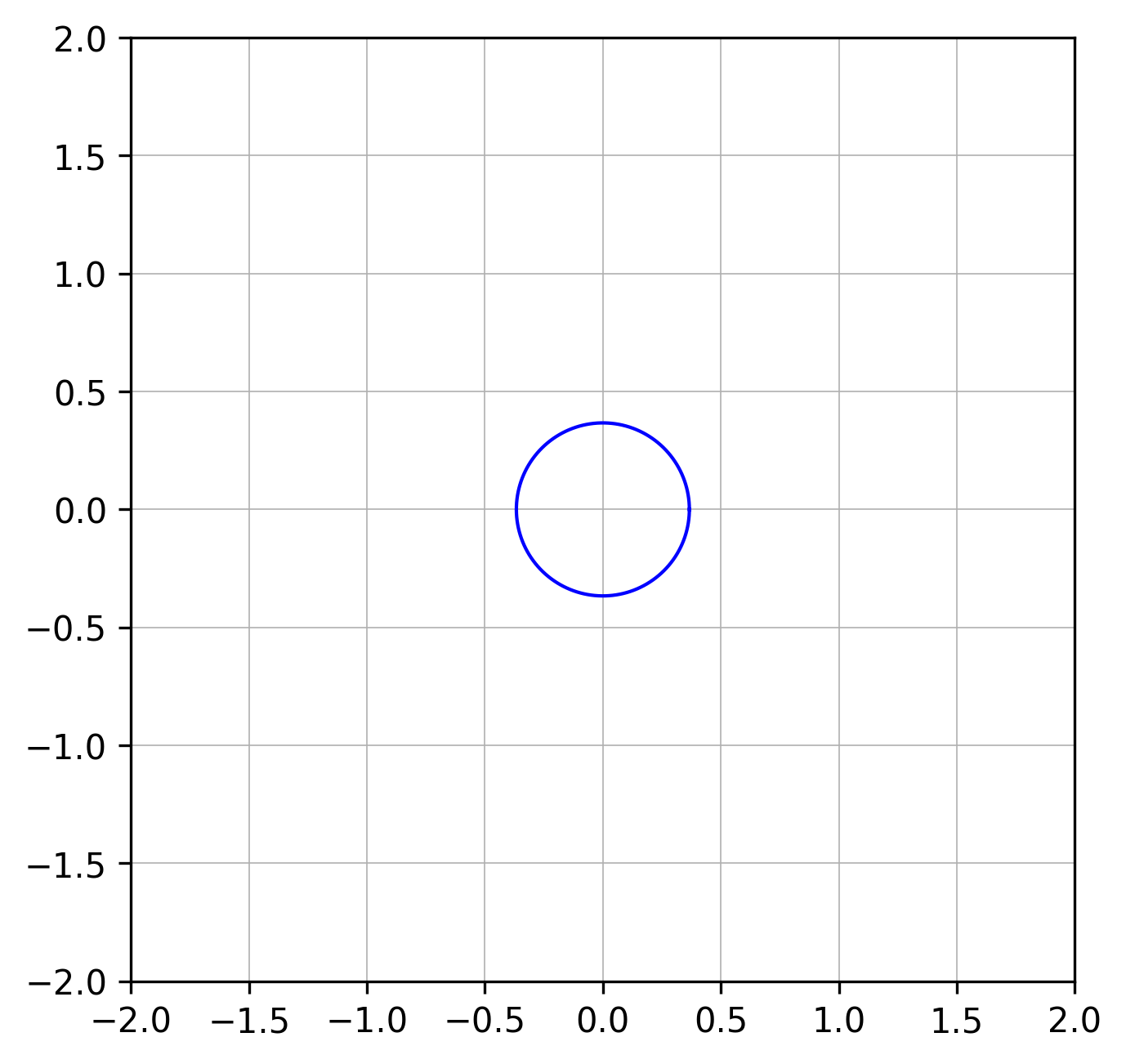}
      			\caption{For $q=0.4$ set $\Gamma_+$ (left), $\Gamma_-$ (middle) and $q\mathbb{T}_{\sqrt{1-q^2}}$ (right).}
                \label{FIG:2}
                \includegraphics[width=0.53\columnwidth]{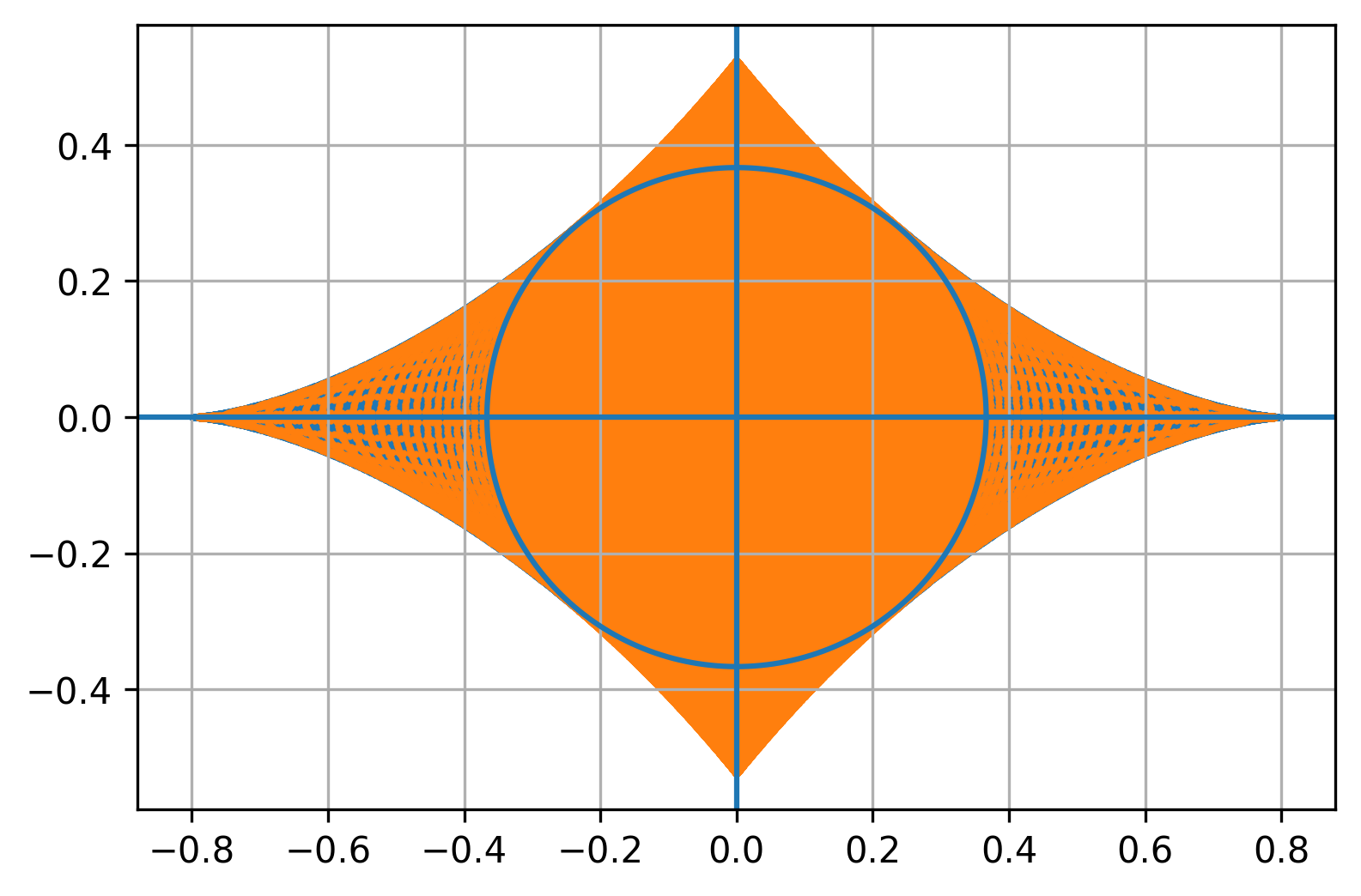}
      			\caption{Ber$_q(T_\phi)=\Gamma_+\cup \Gamma_-\cup q\mathbb{T}_{\sqrt{1-q^2}}$ for $q=0.4$ (apparently not convex).}
                \label{FIG:3}
                \end{figure}
                In particular, if $q=1$ then from the above example we have the following result.
     \begin{corollary}
      	  Ber$(T_\phi)=(-2,2)$ with $\phi(z)=z+\overline{z}$. 
      	\end{corollary}

\begin{example}\label{shift2}
Let us consider the multiplication operator $M_\phi$ acting on $H^2(\mathbb{D})$, where $\phi(z)=p(z)=d_0+d_1z+\cdots+d_kz^k$, $d_i\in\mathbb{C}$ for $i=0,1,\cdots,k$ then from equation~\eqref{multeq} we have, \[\mathrm{Ber}_q(M_{p(z)})=qp(\mathbb{D}).\]
Evidently the convexity of Ber$_q(M_{p(z)})$ depends on the image of $\mathbb{D}$ under the polynomial $p(z)$. We have plotted the $q$-Berezin range of $M_{p(z)}$, where $p(z)=(1+i)(1+z+z^2+\cdots+z^k)$ for $k=1,2,3$ and $q=0.8$.
\end{example}
        
        \begin{figure}[H]
      	\centering
      	\includegraphics[width=0.9\columnwidth]{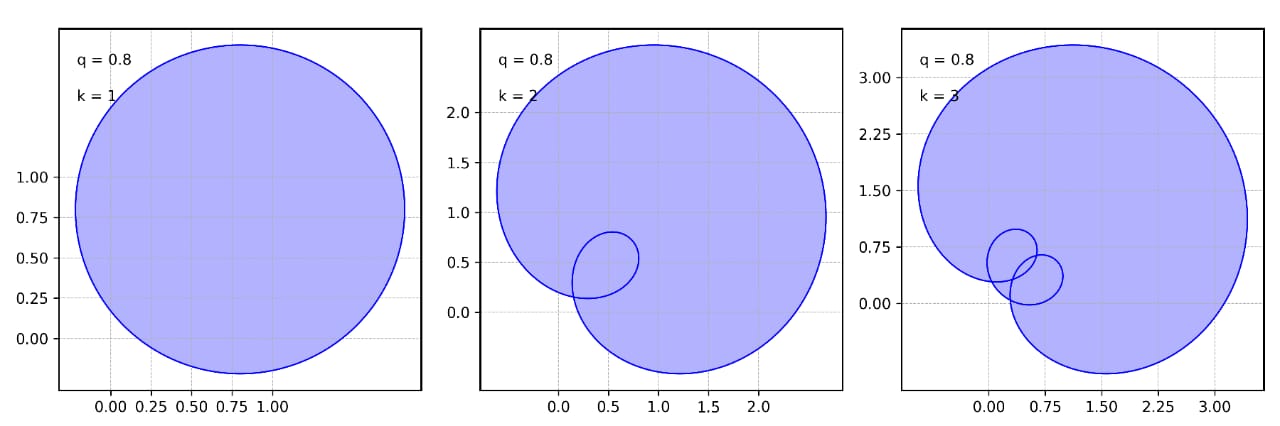}
      	\caption{Ber$_q(M_{p(z)})$ for $q=0.8$ and $p(z)=(1+i)(1+z)$ (left, apparently convex), $p(z)=(1+i)(1+z+z^2)$ (middle, apparently not convex), $p(z)=(1+i)(1+z+z^2+z^3)$ (right, apparently not convex).}
        \label{FIG:4}
      \end{figure}


        \subsection{Weighted shift operators}
        For a bounded weight sequence $\alpha=\{\alpha_n\}_{n=0}^\infty$ the unilateral weighted shift operator $T_\alpha:H^2(\mathbb{D})\to H^2(\mathbb{D})$ is defined by
      	\[T_\alpha(f(z))=\sum_{n=0}^\infty \alpha_na_nz^{n+1}=\sum_{n=0}^\infty\alpha_n\langle f,z^n\rangle z^{n+1},\] where $f(z)=\sum\limits_{n=0}^\infty a_nz^n\in H^2(\mathbb{D})$. With respect to the basis $\{e_n\}$, $e_n(z)=z^n$, $T_\alpha$ has the matrix representation 
     \[
     \begin{pmatrix}
     	0 & 0 & 0 & \cdots  \\
     	\alpha_0 & 0 & 0 & \cdots \\
     	0   & \alpha_1 & 0 & \cdots \\
     	\vdots & \vdots & \vdots & \ddots  \\
     \end{pmatrix}.
     \]
    In the following theorem, we characterize the convexity of Ber$_q(T_\alpha)$ for weights $\alpha_n$.
     \begin{theorem}\label{new ws}
        Let $T_\alpha$ be a unilateral weighted shift operator acting on $H^2(\mathbb{D})$, where $\alpha=\{\alpha_n\}_{n=0}^\infty$ be a bounded real sequence and $0<q\leq1$. Then Ber$_q(T_\alpha)$ is a disc centered at the origin, hence convex. 
     \end{theorem}
     \begin{proof}
        Let $w_1\in\mathbb{D}$ and $w_2 \in \mathcal{S}_{w_1}$, that is, $\langle\hat{k}_{w_1},\hat{k}_{w_2}\rangle=q$. Then
      	\begin{eqnarray}\label{weight equ}
      		\langle T_\alpha\hat{k}_{w_1},\hat{k}_{w_2}\rangle 
      		&=&\sqrt{(1-|w_1|^2)(1-|w_2|^2)}\left\langle \sum_{n=0}^\infty \alpha_n\langle k_{w_1},z^n\rangle z^{n+1},k_{w_2}\right\rangle\nonumber\\
      		&=& q(1-\overline{w_1}w_2)w_2\sum_{n=0}^\infty \alpha_n(\overline{w_1}w_2)^n \quad\text{(using equation~\eqref{neweq})}.
      \end{eqnarray} 
      		 If $w_1=0$ then $\langle T_\alpha\hat{k}_{w_1},\hat{k}_{w_2}\rangle=q\alpha_0 w_2$, where $w_2\in\mathbb{T}_{\sqrt{1-q^2}}$. Now, we have,
      		\begin{eqnarray*}
      		\mathrm{Ber}_q(T_\alpha)&=&\left\{u_\pm(|w_1|)e^{i\theta}:0<|w_1|<1,0\leq\theta<2\pi\right\}\bigcup q\alpha_0\mathbb{T}_{\sqrt{1-q^2}}\\
      		&=& \mathcal{X}_\pm \bigcup q\alpha_0\mathbb{T}_{\sqrt{1-q^2}},
      	\end{eqnarray*}
      		where $\mathcal{X}_\pm = \left\{u_\pm(|w_1|)e^{i\theta}:0<|w_1|<1,0\leq\theta<2\pi\right\},$ and the continuous functions $u_\pm : (0,1) \to \mathbb{R}$ are defined by
            \[u_\pm(|w_1|)=q(1-\lambda^\pm_{w_1}|w_1|^2)\lambda^\pm_{w_1}|w_1|\sum_{n=0}^\infty\alpha_n(\lambda^\pm_{w_1}|w_1|^2)^n.\] Let $\eta\in \mathcal{X}_\pm$. Then
      		\[
      		\eta=u_\pm(|w_1|)e^{i\theta}
      		\] for some $0<|w_1|<1$, $0\leq\theta<2\pi$. Now, for any $\zeta\in[0,2\pi)$,
      		\[
      		\eta e^{i\zeta}=u_\pm(|w_1|)e^{i(\theta+\zeta)}\in \mathcal{X}_\pm,
      		\] since, $\eta e^{i\zeta}$ is the image of $|w_1|e^{i\left(\theta+\zeta\right)}$ for all $\zeta\in[0,2\pi)$, where $0<|w_1|<1$. As, $0\in\mathrm{Ber}_q(T_\alpha)$ and $\lim\limits_{|w_1|\to 0^+}u_\pm(|w_1|)=\pm q\alpha_0\sqrt{1-q^2}$ it is obvious that Ber$_q(T_\alpha)$ is a disc with center at the origin.
     \end{proof}
     Some particular cases of sequence $\{\alpha_n\}_{n=0}^\infty$ are considered below.
        \begin{corollary} \label{shift coro}
        Let $T_\alpha$ be a unilateral weighted shift operator with weights $\alpha_n$. Then we have the following results.
        \begin{enumerate}
            \item [(i)] If $\alpha_n=\alpha_0$, for all $n\geq0$ and for some $\alpha_0\in\mathbb{R}$ then Ber$_q(T_\alpha)$ is a disc centered at origin.
            \item [(ii)]  If $\alpha_n=\beta^n$ for some fixed $\beta\in(0,1)$ and $n\in\mathbb{N}\cup\{0\}$ then Ber$_q(T_\alpha)$ is a disc centered at origin.
            \end{enumerate}
            \end{corollary}
            \begin{proof}
            (i) If $\alpha_n=\alpha_0$, for all $n\geq0$ then from equation~\eqref{weight equ} it directly follows that 
            \[\mathrm{Ber}_q(T_\alpha)=\bigcup_{w_1\in\mathbb{D}}\{q\alpha_0 w_2:w_2\in\mathcal{S}_{w_1}\}=q\alpha_0\mathbb{D}.\]
       (ii) If $\alpha_n=\beta^n$ then from Theorem~\ref{new ws}, we obtain,
\[
\mathrm{Ber}_q(T_\alpha)
=\left\{u_\pm(|w_1|)e^{i\theta}: 0<|w_1|<1,\; 0\leq \theta<2\pi\right\}
\;\cup\; q\mathbb{T}_{\sqrt{1-q^2}},
\]
where the continuous functions $u_\pm:(0,1)\to\mathbb{R}$ are defined by
\[
u_\pm(|w_1|)=\frac{q\bigl(1-\lambda^\pm_{w_1}|w_1|^2\bigr)}{1-\beta \lambda^\pm_{w_1}|w_1|^2}\,\lambda^\pm_{w_1}|w_1|.
\]
Hence, $\mathrm{Ber}_q(T_\alpha)$ is a disc centered at the origin. 
\end{proof}
\begin{remark}
Some findings on above corollary are stated below.
\begin{enumerate}
    \item [(i)] From Corollary~\ref{shift coro} (i) it is easy to see that the $q$-Berezin number is $q\alpha_0$.
    \item [(ii)] From Corollary~\ref{shift coro} (ii) in order to find ber$_q(T_\alpha)$, we have to compute $\max\limits_{|w_1|\in(0,1)}|u_\pm(|w_1|)|$. Consider \[u_+(|w_1|)=\frac{q\bigl(1-\lambda^+_{w_1}|w_1|^2\bigr)}{1-\beta \lambda^+_{w_1}|w_1|^2}\,\lambda^+_{w_1}|w_1|,
\] where $\lambda^+_{w_1}|w_1|=\dfrac{(1-k^2)|w_1|+(1-|w_1|^2)k}{1-k^2|w_1|^2}=\dfrac{|w_1|+k}{1+k|w_1|},\ k=\sqrt{1-q^2}$. Let $|w_1|=r$. Now,
            \[u_+(r)=\dfrac{\sqrt{1-k^2}\left(1-\tfrac{r(r+k)}{1+kr}\right)}{1-\beta\tfrac{r(r+k)}{1+kr}}\left(\dfrac{r+k}{1+kr}\right)
            =\dfrac{\sqrt{1-k^2}(1-r^2)(r+k)}{(1+kr-\beta r^2-k\beta r)(1+kr)}.\] Taking logarithm on both sides we have, 
            \[F_+(r)=\log u_+(r)=\log\sqrt{1-k^2}+\log(1-r^2)+\log(r+k)-\log (1+kr-\beta r^2-k\beta r)-\log(1+kr).\] To detect the extreme points, it is enough to derivate $F_+(r)$ and equate to zero. 
Now,
\begin{eqnarray*}
F_+'(r) &=&
-2r(r+k)(1+kr)(1+kr-\beta r^2-k\beta r)+(1-r^2)(1+kr)(1+kr-\beta r^2-k\beta r)-\\
&&(k-2\beta r-k\beta)(1-r^2)(r+k)(1+kr)-k(1-r^2)(r+k)(1+kr-\beta r^2-k\beta r).
\end{eqnarray*}
If there exists $r_0\in(0,1)$ such that $F_+'(r_0)=0$ and $F_+''(r_0)<0$, then $u_+(r)$ attains a maximum at $r_0$. Similarly, considering $u_-(r)$, if there exists $r_1\in(0,1)$ such that $F_-'(r_1)=0$ and $F_-''(r_1)<0$, then $u_-(r)$ attains a maximum at $r_1$. Consequently,
\[
\mathrm{ber}_q(T_\alpha)
=\max\left\{u_+(r_0),\,u_-(r_1),\,q\sqrt{1-q^2}\right\}.
\]
\end{enumerate}
In particular, for $q=0.5$ and $\beta=0.5$, we have, \[\mathrm{ber}_q(T_\alpha)=q\sqrt{1-q^2}=0.4330.\]
\end{remark}
      
\subsection{Composition operators}

We investigate the convexity of the \(q\)-Berezin range associated with composition operators on the Hardy space \(H^2(\mathbb{D})\). Let \(\phi:\mathbb{D}\to\mathbb{D}\) be an analytic self-map, and consider the composition operator \(C_\phi\) defined by \(C_\phi f = f \circ \phi\). In this section, we restrict our attention to the case where \(\phi\) is an elliptic-type symbol as well as a Blaschke factor.

First, we focus on the case $\phi(z) = \xi z$ with $\xi \in \overline{\mathbb{D}}$. Let \(w_1 \in \mathbb{D}\) and \(w_2 \in \mathcal{S}_{w_1}\), that is, \(\langle \hat{k}_{w_1}, \hat{k}_{w_2} \rangle = q\). Then,
\[
\langle C_{\phi}\hat{k}_{w_1}, \hat{k}_{w_2} \rangle
= \sqrt{(1-|w_1|^2)(1-|w_2|^2)} \, C_\phi k_{w_1}(w_2)
= \frac{q(1-\overline{w_1}w_2)}{1-\xi(\overline{w_1}w_2)},
\]
where the final equality follows from equation~\eqref{neweq}. In the particular case \(w_1 = 0\), this reduces to \(\langle C_{\phi}\hat{k}_{w_1}, \hat{k}_{w_2} \rangle = q\).

Consequently, by Lemma~\ref{range}, the \(q\)-Berezin range of \(C_\phi\) is given by
\begin{equation}\label{newcomp}
\mathrm{Ber}_q(C_\phi)
=
\left\{
\frac{q(1-t)}{1-\xi t}
:\;
t \in \left[\frac{q-1}{q+1},\,1\right)
\right\}.
\end{equation}

In this setting we have the following lemma.                
                \begin{lemma}\label{comp lemma}
                Let $\phi(z) = \xi z$ with $\xi \in \overline{\mathbb{D}}$. Then
                     \begin{enumerate}
        \item [(i)] Ber$_q(C_\phi)$ is a singleton if and only if $\xi=1$,
        \item [(ii)] $\Im\left\{\mathrm{Ber}_q(C_\phi)\right\}=\{0\}$ if and only if $\Im\{\xi\}=0$.
    \end{enumerate}
                \end{lemma}
                \begin{proof}
        (i) As, the sufficient part is clear, we only show the necessary part. Let Ber$_q(C_\phi)$ is a singleton. Since, $\dfrac{q(1-t)}{1-\xi t}=q$ at $t=0$ then we have, $\mathrm{Ber}_q(C_\phi)=\{q\},$ and consequently
        \[\dfrac{q(1-t)}{1-\xi t}=q \ \mbox{holds for all } t\in\left[\dfrac{q-1}{q+1},1\right).\]
        This implies $\xi t=t$ for all $t\in\left[\dfrac{q-1}{q+1},1\right)\setminus\{0\}$. Hence, $\xi=1$.

        (ii) If $\Im\{\xi\}=0$ then it is easy to see that $\Im\left\{\mathrm{Ber}_q(C_\phi)\right\}=\{0\}$. Conversely, let $\Im\left\{\mathrm{Ber}_q(C_\phi)\right\}=\{0\}$. Then $\dfrac{q(1-t)}{1-\xi t}\in\mathbb{R}$ for all $t\in\left[\dfrac{q-1}{q+1},1\right)$. In particular, this holds for all $t\neq0$ as, the case $t=0$ is trivial. Let $\dfrac{q(1-t)}{1-\xi t}=h_t$, for some $h_t\in\mathbb{R}\setminus\{0\}$. Then for $t\neq0$, $\xi=\dfrac{1}{t}-\dfrac{q(1-t)}{th_t}\in\mathbb{R}$. Hence, $\Im\{\xi\}=0$. 
        \end{proof}

The mapping
\[
w = \frac{q(1-t)}{1-\xi t}, \qquad t \in \left[\frac{q-1}{q+1},\,1\right),
\]
is a Möbius transformation of the real variable $t$. It is well known that Möbius transformations map line-segments in the extended real line to either line-segments or circular arcs in the complex plane. Hence, Ber$_q(C_\phi)$ is either a circular arc or a line-segment. Consequently, the problem of determining the convexity of $\mathrm{Ber}_q(C_\phi)$ reduces to identifying those values of $\xi$ for which the image of the interval $\left[\dfrac{q-1}{q+1},\,1\right)$ under $w$ is a line-segment in the complex plane. In this regard, we have the following result.
		\begin{theorem}\label{comp}
Let $\phi(z)=\xi z$, where $\xi \in \overline{\mathbb{D}}$, and let $0<q\leq 1$. Then the following statements are equivalent:
\begin{enumerate}
    \item[(i)] The set $\mathrm{Ber}_q(C_\phi)$ is convex.
    \item[(ii)] The set $\mathrm{Ber}_q(C_\phi)$ is a (possibly degenerate) line-segment.
    \item[(iii)] $\xi \in [-1,1]$.
\end{enumerate}
\end{theorem}
		\begin{proof}
        It is enough to prove that $\mathrm{Ber}_q(C_\phi)$ is a line-segment if and only if $-1\leq\xi\leq1.$ Let $w_1\in\mathbb{D}$ and $w_2\in \mathcal{S}_{w_1}$, that is, $\langle \hat{k}_{w_1},\hat{k}_{w_2}\rangle=q$. Then from equation~\eqref{newcomp} we obtain,
		\[\mathrm{Ber}_q(C_\phi)=\left\{\dfrac{q(1-t)}{1-\xi t}:t\in\left[\dfrac{q-1}{q+1},1\right)\right\}.\]
			 Firstly suppose that $\xi=1$ then  
			Ber$_q(C_{\phi})=\{q\}$, which is convex. Similarly, for $-1\leq\xi<1$, 
 $\dfrac{q(1-t)}{1-\xi t}$ is a decreasing function in $t$. Now $\lim\limits_{t\to1^-}\dfrac{q(1-t)}{1-\xi t}=0$ and \[\max\left\{\dfrac{q(1-t)}{1-\xi t}:t\in\left[\dfrac{q-1}{q+1},1\right)\right\}=\dfrac{2q}{q+1-\xi q+\xi}.\] 
			 Hence, Ber$_q(C_{\phi})=\left(0,\dfrac{2q}{q+1-\xi q+\xi}\right]$, which is also a line-segment.
		\noindent	Conversely, suppose that 
			\[	\mathrm{Ber}_q(C_{\phi})
				=\left\{\dfrac{q(1-t)}{1-\xi t}:t\in\left[\dfrac{q-1}{q+1},1\right)\right\}\]
		  is a line-segment or a singleton set. From Lemma~\ref{comp lemma}~(i), Ber$_q(C_{\phi})$ is a point $\{q\}$ if and only if $\xi=1$. Let $w=\dfrac{q(1-t)}{1-\xi t}$. Clearly it is a M\"obius transformation of the real variable $t$. This implies $t=\dfrac{q-w}{q-\xi w}$. Now, $t=\overline{t}$ gives
        \[q(\overline{w}-w)+q(\xi w-\overline{\xi}\overline{w})+|w|^2(\overline{\xi}-\xi)=0.\]
        Since, Ber$_q(C_\phi)$ is a line-segment we must have,
        $\overline{\xi}-\xi=0$, which follows that $\Im\{\xi\}=0$. As, $\xi\in\overline{\mathbb{D}}$, we have $-1\leq\xi\leq1$.
		\end{proof}
\begin{remark}
Some observations on Theorem \ref{comp} are mentioned below.
\begin{enumerate}
\item[(i)] If $\Im(\xi) \neq 0$ then Ber$_q(C_{\phi})$ lies on the circle
\[q(\overline{z}-z)+q(\xi z-\overline{\xi}\overline{z})+|z|^2(\overline{\xi}-\xi)=0\]
in the complex plane.
    \item[(ii)] Let $\xi \in \mathbb{T}$ and $\phi(z)=\xi z$. For $0<q\leq 1$, the $q$-Berezin range of $C_\phi$ on $H^2(\mathbb{D})$ is convex if and only if either $\xi=1$ or $\xi=-1$.

    \item[(iii)] For $q=1$, the Berezin range $\mathrm{Ber}(C_\phi)$ is convex if and only if $-1 \leq \xi \leq 1$, where $\phi(z)=\xi z$ and $\xi \in \overline{\mathbb{D}}$; see \cite[Theorem~4.1]{augustine2023composition}.
\end{enumerate}
      
      \end{remark}

	To this end, we will see that Ber$_q(C_\phi)$ is not always convex. We sketch Ber$_q(C_\phi)$ on $H^2(\mathbb{D})$ for $\phi(z)=\dfrac{i\pi}{4}z$, choosing some fixed values of $q$.
			\begin{figure}[H]
			\centering
			\includegraphics[width=0.5\columnwidth]{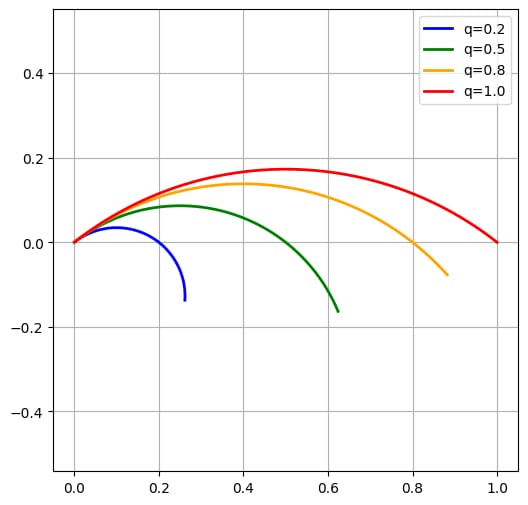}
			\caption{Ber$_q(C_{\phi})$ on $H^2(\mathbb{D})$ for $\xi=\frac{i\pi}{4}$ and $q=0.2,0.5,0.8,1.0$ (apparently not convex).}
            \label{FIG:5}
		\end{figure}

		For $\alpha\in\mathbb{D}$, consider the automorphism of the unit disc $\phi_\alpha(z)=\dfrac{z-\alpha}{1-\overline{\alpha}z}$ and the composition operator $C_{\phi_\alpha}$ acting on $H^2(\mathbb{D})$ is defined by $C_{\phi_\alpha}f=f\circ\phi_\alpha$. In the following result we give some geometrical properties of the $q$-Berezin range of $C_{\phi_\alpha}$ on $H^2(\mathbb{D})$.
		\begin{proposition}\label{conjugation}
			The $q$-Berezin range of $C_{\phi_\alpha}$ on $H^2(\mathbb{D})$ is symmetric about the real line where $0<q\leq1$.
		\end{proposition}
		\begin{proof}
			Let $w_1=r_1e^{i\theta_1},w_2\in \mathbb{D}$ be such that $\langle\hat{k}_{w_1},\hat{k}_{w_2}\rangle=q$, and $\alpha=\rho e^{i\psi}$. Then 
			\begin{eqnarray*}
				\langle C_{\phi_{\alpha}}\hat{k}_{w_1},\hat{k}_{w_2}\rangle 
					=\sqrt{(1-|w_1|^2)(1-|w_2|^2)}(k_{w_1}\circ\phi_\alpha)(w_2)
					=\dfrac{q(1-\overline{w_1}w_2)}{1-\overline{w_1}\phi_\alpha(w_2)} \quad \text{(using equation~\eqref{neweq})}.
			\end{eqnarray*}  Our claim is that 
			\begin{eqnarray*}
				\langle C_{\phi_\alpha}\hat{k}_{w_1},\hat{k}_{w_2}\rangle=\overline{\langle C_{\phi_\alpha}\hat{k}_{\overline{w_1}e^{2i\psi}},\hat{k}_{\overline{w_2}e^{2i\psi}}\rangle}.
			\end{eqnarray*} 
			
			\noindent Case 1: For $w_1=0$, $\langle C_{\phi_\alpha}\hat{k}_{w_1},\hat{k}_{w_2}\rangle=\overline{\langle C_{\phi_\alpha}\hat{k}_{\overline{w_1}e^{2i\psi}},\hat{k}_{\overline{w_2}e^{2i\psi}}\rangle}=q$.
			
			\noindent Case 2: For $w_1\neq0$, we have, $w_2=\lambda^\pm_{w_1}w_1$. Now,
			\begin{eqnarray*}
            &&\langle C_{\phi_\alpha}\hat{k}_{w_1},\hat{k}_{w_2}\rangle=\overline{\langle C_{\phi_\alpha}\hat{k}_{\overline{w_1}e^{2i\psi}},\hat{k}_{\overline{w_2}e^{2i\psi}}\rangle}\\
					&\iff& \dfrac{q(1-\lambda^\pm_{w_1}r_1^2)}{1-r_1e^{-i\theta_1}\phi_\alpha(\lambda^\pm_{w_1}r_1e^{i\theta_1})}=\overline{\dfrac{q(1-r_1e^{-i(2\psi-\theta_1)}\lambda^\pm_{w_1}r_1e^{i(2\psi-\theta_1)})}{1-r_1e^{-i(2\psi-\theta_1)}\phi_\alpha(\lambda^\pm_{w_1}r_1e^{i(2\psi-\theta_1)})}}\\
					&\iff& r_1e^{-i\theta_1}\phi_\alpha(\lambda^\pm_{w_1}r_1e^{i\theta_1})=r_1e^{i(2\psi-\theta_1)}\overline{\phi_\alpha(\lambda^\pm_{w_1}r_1e^{i(2\psi-\theta_1)})}\\
                    &\iff& \phi_\alpha(\lambda^\pm_{w_1}r_1e^{i\theta_1})=e^{i2\psi}\overline{\phi_\alpha(\lambda^\pm_{w_1}r_1e^{i(2\psi-\theta_1)})}.
			\end{eqnarray*}
            Finally,
			\begin{eqnarray*}
					e^{i2\psi}\overline{\phi_\alpha(\lambda^\pm_{w_1}r_1e^{i(2\psi-\theta_1)})}=e^{i2\psi}\dfrac{\lambda^\pm_{w_1}r_1e^{i(\theta_1-2\psi)}-\rho e^{-i\psi}}{1-\rho e^{i\psi}\lambda^\pm_{w_1}r_1e^{i(\theta_1-2\psi)}}
					 = \dfrac{\lambda^\pm_{w_1}r_1e^{i\theta_1}-\rho e^{i\psi}}{1-\rho e^{-i\psi}\lambda^\pm_{w_1}r_1e^{i\theta_1}}
					=\phi_\alpha(\lambda^\pm_{w_1}r_1e^{i\theta_1}).
			\end{eqnarray*}
		\end{proof}
		\begin{corollary}\label{comp coro}
			If the $q$-Berezin range of $C_{\phi_\alpha}$ on $H^2(\mathbb{D})$ is convex, then $\Re\left\{\langle C_{\phi_\alpha}\hat{k}_{w_1},\hat{k}_{w_2}\rangle\right\}\in$ Ber$_q(C_{\phi_\alpha})$ for each $w_1\in \mathbb{D}$, and $w_2\in\mathcal{S}_{w_1}$.
		\end{corollary}
		\begin{proof}
			Suppose, Ber$_q(C_{\phi_\alpha})$ is convex. Also, from Proposition~\ref{conjugation} since, it is symmetric about the real line, we have,
			\[	\dfrac{1}{2}\langle C_{\phi_\alpha}\hat{k}_{w_1},\hat{k}_{w_2}\rangle+\dfrac{1}{2}\overline{\langle C_{\phi_\alpha}\hat{k}_{\overline{w_1}e^{2i\psi}},\hat{k}_{\overline{w_2}e^{2i\psi}}\rangle}=\Re\left\{\langle C_{\phi_\alpha}\hat{k}_{w_1},\hat{k}_{w_2}\rangle\right\}\in \mathrm{Ber}_q(C_{\phi_\alpha}).\]
		\end{proof}
        Let $w_1\in\mathbb{D}$ and $w_2\in \mathcal{S}_{w_1}$, that is, $\langle\hat{k}_{w_1},\hat{k}_{w_2}\rangle=q$. Then for the automorphism $\phi_\alpha(z)=\dfrac{z-\alpha}{1-\overline{\alpha}z},\ z\in\mathbb{D}$, we have,
        \begin{eqnarray}\label{eq comp}
		\langle C_{\phi_{\alpha}}\hat{k}_{w_1},\hat{k}_{w_2}\rangle &=&\dfrac{q(1-\overline{w_1}w_2)}{1-\overline{w_1}\phi_\alpha(w_2)}\nonumber\\
        &=&\dfrac{q(1-\overline{w_1}w_2)(1-\overline{\alpha}w_2)}{1-\overline{w_1}w_2-\overline{\alpha}w_2+\alpha\overline{w_1}}.
		\end{eqnarray}
        Now, we plot Ber$_q(C_{\phi_\alpha})$ for $\alpha=-\dfrac{1}{2}$ and $q=0.5$, which can be expressed as \[\mathrm{Ber}_q(C_{\phi_\alpha})=\Delta_+\cup \Delta_-\cup \{q\},\] where 
\[\Delta_+=\left\{\dfrac{q(1-\lambda^+_{w_1}|w_1|^2)\left(1-\overline{\alpha}\lambda^+_{w_1}|w_1|e^{i\theta}\right)}{1-\lambda^+_{w_1}|w_1|^2-\overline{\alpha}\lambda^+_{w_1}|w_1|e^{i\theta}+\alpha|w_1|e^{-i\theta}}:0<|w_1|<1,0\leq\theta<2\pi\right\},\] and \[\Delta_-=\left\{\dfrac{q(1-\lambda^-_{w_1}|w_1|^2)\left(1-\overline{\alpha}\lambda^-_{w_1}|w_1|e^{i\theta}\right)}{1-\lambda^-_{w_1}|w_1|^2-\overline{\alpha}\lambda^-_{w_1}|w_1|e^{i\theta}+\alpha|w_1|e^{-i\theta}}:0<|w_1|<1,0\leq\theta<2\pi\right\}.\] 
			\begin{figure}[H]
		\centering
		\includegraphics[width=0.42\columnwidth]{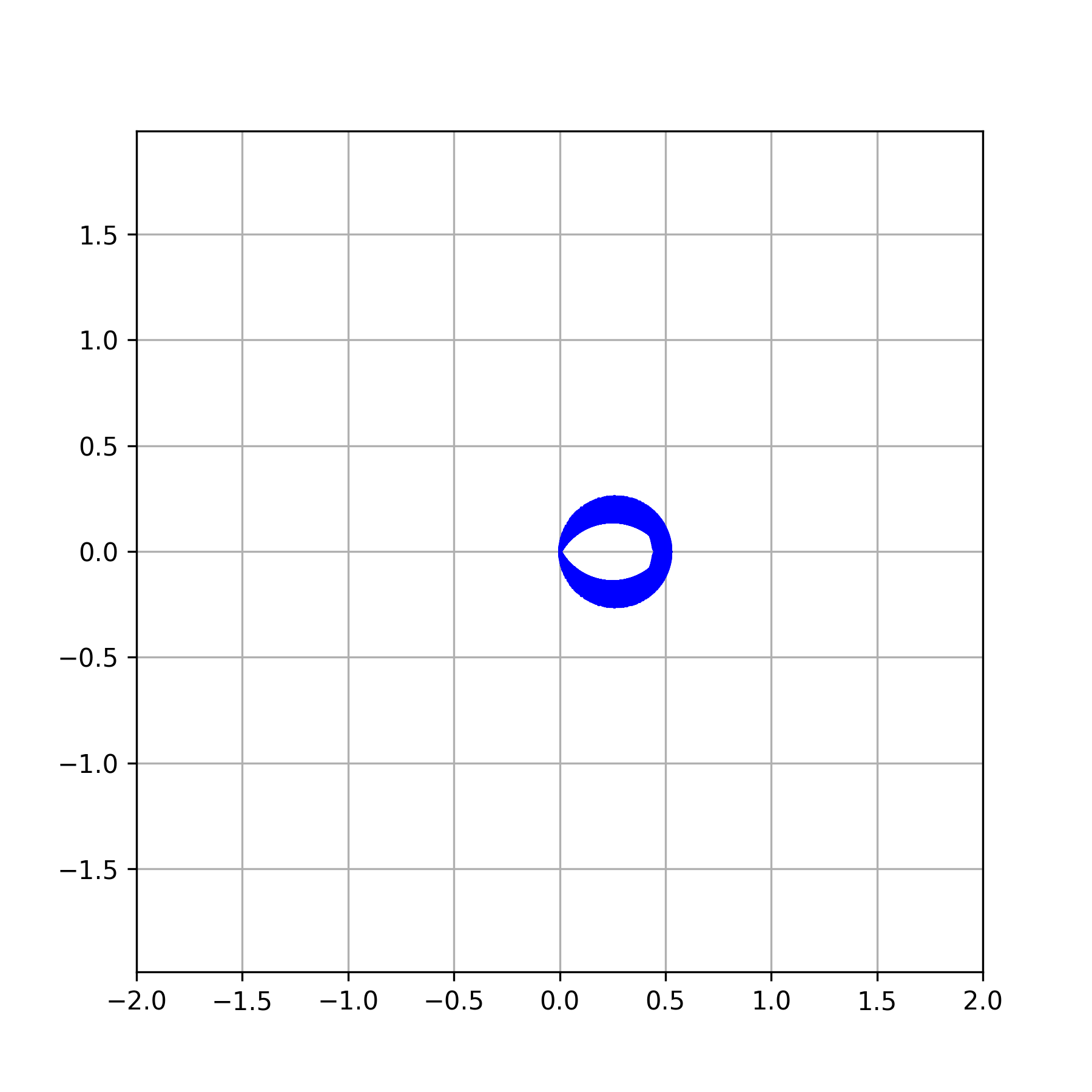}
		\includegraphics[width=0.42\columnwidth]{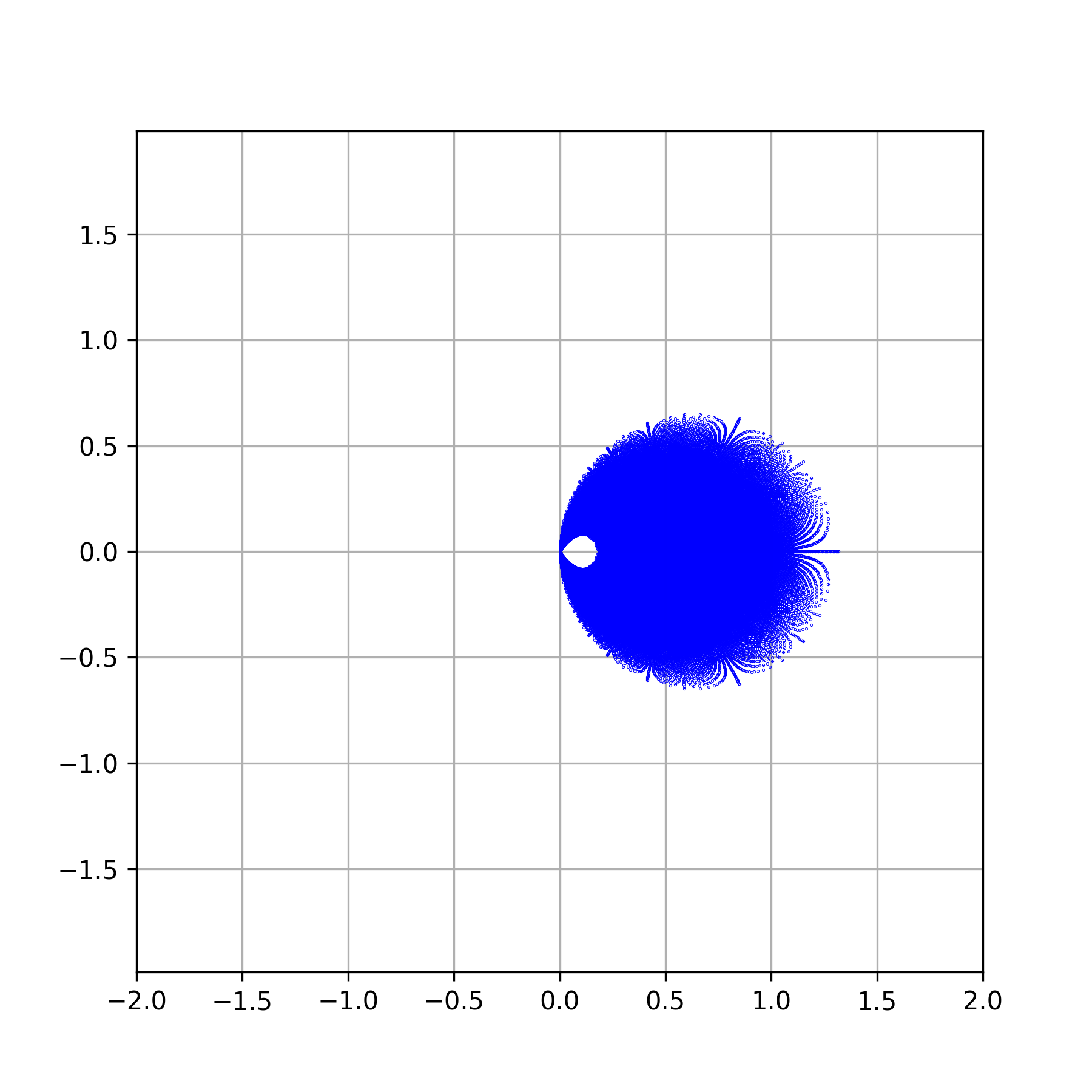}
		\caption{$\Delta_+$ (left) and $\Delta_-$ (right) for $\alpha=-\frac{1}{2}$ and $q=0.5$.}
        \label{FIG:6}
	\end{figure}
    From the above figures it is clear Ber$_q(C_{\phi_\alpha})$ is not convex in general. In the next theorem we will give characterization of $\alpha\in\mathbb{D}$ such that Ber$_q(C_{\phi_\alpha})$ is convex.
\begin{theorem}
	The $q$-Berezin range of $C_{\phi_\alpha}$ on $H^2(\mathbb{D})$ is convex if and only if $\alpha=0$.
	\end{theorem}
	\begin{proof}
			If $\alpha=0$ then from equation~\eqref{eq comp} we get, Ber$_q(C_{\phi_\alpha})=\{q\}$, which is convex.

            Conversely, suppose that Ber$_q(C_{\phi_\alpha})$ is convex. 
            Let $w_1 (\neq0)\in\mathbb{D}$ and $w_2\in\mathcal{S}_{w_1}$, that is, $w_2=\lambda^\pm_{w_1}w_1$. Thus, equation~\eqref{eq comp} reduces to 
            \begin{eqnarray}\label{eq red}
\langle C_{\phi_{\alpha}}\hat{k}_{w_1},\hat{k}_{w_2}\rangle=\dfrac{q(1-\lambda^\pm_{w_1}|w_1|^2)(1-\overline{\alpha}\lambda^\pm_{w_1}w_1)}{1-\lambda^\pm_{w_1}|w_1|^2-\overline{\alpha}\lambda^\pm_{w_1}w_1+\alpha\overline{w_1}}.\end{eqnarray} Now, we break it into real and imaginary parts. Let $\lambda^\pm_{w_1}|w_1|^2=t$, $\overline{\alpha}w_1=z=x+iy$. Hence, $\alpha\overline{w_1}=\overline{z}$ and $\overline\alpha w_2=\lambda^\pm_{w_1}z$. After substituting these into equation~\eqref{eq red} and carrying out straightforward computations, we obtain,
\[\Re\left\{\langle C_{\phi_\alpha}\hat{k}_{w_1},\hat{k}_{w_2}\rangle\right\}=q\dfrac{(1-t)[(1-\lambda^\pm_{w_1}x)(1-t-\lambda^\pm_{w_1}x+x)+\lambda^\pm_{w_1}y^2(\lambda^\pm_{w_1}-1)]}{(1-t-\lambda^\pm_{w_1}x+x)^2+y^2(\lambda^\pm_{w_1}-1)^2}\] and 
\[\Im\left\{\langle C_{\phi_\alpha}\hat{k}_{w_1},\hat{k}_{w_2}\rangle\right\}=q\dfrac{(1-t)y[(1-\lambda^\pm_{w_1}x)(\lambda^\pm_{w_1}-1)-\lambda^\pm_{w_1}(1-t-\lambda^\pm_{w_1}x+x)]}{(1-t-\lambda^\pm_{w_1}x+x)^2+y^2(\lambda^\pm_{w_1}-1)^2}.\] Corollary~\ref{comp coro} gives $\Re\left\{\langle C_{\phi_\alpha}\hat{k}_{w_1},\hat{k}_{w_2}\rangle\right\}\in\mathrm{Ber}_q(C_{\phi_\alpha})$. This implies for each $w_1\in\mathbb{D}\setminus\{0\}$ and $w_2\in\mathcal{S}_{w_1}$ we can get $z_1\in\mathbb{D}\setminus\{0\}$ and $z_2\in\mathcal{S}_{z_1}=\left\{z_2\in\mathbb{D}:\langle\hat{k}_{z_1},\hat{k}_{z_2}\rangle=q\right\}$ such that 
\[\langle C_{\phi_\alpha}\hat{k}_{z_1},\hat{k}_{z_2}\rangle=\Re\left\{\langle C_{\phi_\alpha}\hat{k}_{w_1},\hat{k}_{w_2}\rangle\right\}.\] This follows that \[\Im\left\{\langle C_{\phi_\alpha}\hat{k}_{z_1},\hat{k}_{z_2}\rangle\right\}=q\dfrac{(1-t_1)v[(1-\lambda^\pm_{z_1}u)(\lambda^\pm_{z_1}-1)-\lambda^\pm_{z_1}(1-t_1-\lambda^\pm_{z_1}u+u)]}{(1-t_1-\lambda^\pm_{z_1}u+u)^2+v^2(\lambda^\pm_{z_1}-1)^2}=0.\] Since, $t_1=\lambda^\pm_{z_1}|z_1|^2\in\left[\dfrac{q-1}{q+1},1\right)$, and $(1-\lambda^\pm_{z_1}u)(\lambda^\pm_{z_1}-1)-\lambda^\pm_{z_1}(1-t_1-\lambda^\pm_{z_1}u+u)]=\lambda^\pm_{z_1}t_1-1<0$, where $\overline{\alpha}z_1=u+iv$, then we have $\Im\left\{\langle C_{\phi_\alpha}\hat{k}_{z_1},\hat{k}_{z_2}\rangle\right\}=0$ if and only if $v=\Im\{\overline{\alpha}z_1\}=0$. This tells that $\alpha$ and $z_1$ lie on a line passing through the origin. Hence, $z_1=p\alpha$ for some $p\in\left(-\dfrac{1}{|\alpha|},\dfrac{1}{|\alpha|}\right)$. Now, we have, 
\begin{eqnarray*}
    \langle C_{\phi_\alpha}\hat{k}_{z_1},\hat{k}_{z_2}\rangle&=&\Re\left\{\langle C_{\phi_\alpha}\hat{k}_{p\alpha},\hat{k}_{\lambda^\pm_{p\alpha}p\alpha}\rangle\right\}\\
    &=& q\dfrac{(1-\lambda^\pm_{p\alpha}p^2|\alpha|^2)(1-\lambda^\pm_{p\alpha}p|\alpha|^2)}{1-\lambda^\pm_{p\alpha}p^2|\alpha|^2-\lambda^\pm_{p\alpha}p|\alpha|^2+p|\alpha|^2}.
\end{eqnarray*} Let $k=\sqrt{1-q^2}$. 
Consequently, 
\[\left\{\sqrt{1-k^2}\dfrac{(1-\lambda^+_{p\alpha}p^2|\alpha|^2)(1-\lambda^+_{p\alpha}p|\alpha|^2)}{1-\lambda^+_{p\alpha}p^2|\alpha|^2-\lambda^+_{p\alpha}p|\alpha|^2+p|\alpha|^2}:p\in\left(-\dfrac{1}{|\alpha|},\dfrac{1}{|\alpha|}\right)\right\}=\left(\dfrac{\sqrt{1-k^2}(1-|\alpha|)}{1-|\alpha|k},\dfrac{\sqrt{1-k^2}(1+|\alpha|)}{1+|\alpha|k}\right),\] and 
\[\left\{\sqrt{1-k^2}\dfrac{(1-\lambda^-_{p\alpha}p^2|\alpha|^2)(1-\lambda^-_{p\alpha}p|\alpha|^2)}{1-\lambda^-_{p\alpha}p^2|\alpha|^2-\lambda^-_{p\alpha}p|\alpha|^2+p|\alpha|^2}:p\in\left(-\dfrac{1}{|\alpha|},\dfrac{1}{|\alpha|}\right)\right\}=\left(\dfrac{\sqrt{1-k^2}(1-|\alpha|)}{1+|\alpha|k},\dfrac{\sqrt{1-k^2}(1+|\alpha|)}{1-|\alpha|k}\right).\] Hence, \begin{eqnarray*} 
\langle C_{\phi_\alpha}\hat{k}_{z_1},\hat{k}_{z_2}\rangle&=&\left(\dfrac{\sqrt{1-k^2}(1-|\alpha|)}{1-|\alpha|k},\dfrac{\sqrt{1-k^2}(1+|\alpha|)}{1+|\alpha|k}\right)\bigcup\left(\dfrac{\sqrt{1-k^2}(1-|\alpha|)}{1+|\alpha|k},\dfrac{\sqrt{1-k^2}(1+|\alpha|)}{1-|\alpha|k}\right)\\
&=& \left(
\dfrac{\sqrt{1-k^2}(1-|\alpha|)}{1+|\alpha|k},\dfrac{\sqrt{1-k^2}(1+|\alpha|)}{1-|\alpha|k}
\right).
\end{eqnarray*}
 Putting $w_1=re^{i\theta}$ from equation~\eqref{eq red} one can easily check that 
\[
\lim_{r\to1^-}\langle C_{\phi_{\alpha}}\hat{k}_{w_1},\hat{k}_{w_2}\rangle
=
\begin{cases}
0, & \alpha \neq 0,\\
\sqrt{1-k^2}, & \alpha = 0.
\end{cases}
\] This says that when $\alpha\neq0$ given $\epsilon$ with $0<\epsilon<\dfrac{\sqrt{1-k^2}(1-|\alpha|)}{1+|\alpha|k}$ there exist $w_1\in\mathbb{D}\setminus\{0\}$, $w_2\in\mathcal{S}_{w_1}$ such that $\left|\Re\left\{C_{\phi_{\alpha}}\hat{k}_{w_1},\hat{k}_{w_2}\rangle\right\}\right|<\epsilon$. Also, if $\langle C_{\phi_\alpha}\hat{k}_{z_1},\hat{k}_{z_2}\rangle=\Re\left\{\langle C_{\phi_\alpha}\hat{k}_{w_1},\hat{k}_{w_2}\rangle\right\}$, this is a contradiction as, $\langle C_{\phi_\alpha}\hat{k}_{z_1},\hat{k}_{z_2}\rangle\in\left(
\dfrac{\sqrt{1-k^2}(1-|\alpha|)}{1+|\alpha|k},\dfrac{\sqrt{1-k^2}(1+|\alpha|)}{1-|\alpha|k}
\right)$. Hence, Ber$_q(C_{\phi_\alpha})$ cannot be convex unless $\alpha=0$.  
	\end{proof}
\begin{remark}
    In particular, if $q=1$ then from the above theorem we have Ber$(C_{\phi_\alpha})$ is convex if and only if $\alpha=0$ \cite[Theorem~4.5]{cowen2022convexity}.
\end{remark}		
\section{Conclusion}\label{sec4}
This paper investigates the $q$-Berezin range of several classes of bounded linear operators on the Hardy space $H^2(\mathbb{D})$ for $0<q\leq1$, and establishes results concerning their geometric structure and convexity properties. These results deepen the theoretical understanding of the $q$-Berezin framework, reveal new aspects of the operator-theoretic geometry under kernel constraints. Future work may include establishing these results for other reproducing kernel Hilbert spaces, exploring finer geometric properties, and studying the $q$-Berezin number. 








\section*{Declarations}

\begin{itemize}
	\item Availability of data and materials: Not applicable.
	\item Competing interests: The authors declare that they have no competing interests.
	\item Funding: Not applicable.
	\item Authors' contributions: The authors contribute equally to this work.
	\end{itemize}


\begin{thebibliography}{99}
		\bibitem{augustine2023composition}
			Augustine, A., Garayev, M., Shankar, P.: Composition operators, convexity of their Berezin range and related questions. Complex Anal. Oper. Theory $\mathbf{17}$(8), 126 (2023)
			
	\bibitem{augustine2024berezin}
		Augustine, A., Garayev, M., Shankar, P.: On the Berezin range and the Berezin radius of some operators. arXiv preprint arXiv:2411.10771 (2024)
			
			
			
		\bibitem{berezin1972covariant} Berezin, F. A.: Covariant and contravariant symbols of operators. Math. USSR-Izv. $\mathbf{6}$(5), 1117 (1972)
			
		\bibitem{bhunia2023new} Bhunia, P., Gürdal, M., Paul, K., Sen, A., Tapdigoglu, R.: On a new norm on the space of reproducing kernel Hilbert space operators and Berezin radius inequalities. Numer. Funct. Anal. Optim. $\mathbf{44}$(9), 970–986 (2023)
			
			
			
			
		\bibitem{cowen2022convexity} Cowen, C. C., Felder, C.: Convexity of the Berezin range. Linear Algebra Appl. $\mathbf{647}$, 47–63 (2022)
		
		\bibitem{englivs1995toeplltz} Engli{\v{s}}, M.: Toeplltz operators and the Berezin transform on H2. Linear Algebra Appl. $\mathbf{223}$, 171--204 (1995)
        
        \bibitem{garayev2023weighted} Garayev, M., Bakherad, M., Tapdigoglu, R.: The weighted and the Davis-Wielandt Berezin number. Oper. Matrices $\mathbf{17}$(2), 469–484 (2023)
		
		\bibitem{gau2021numerical} Gau, H.L., Wu, P.Y.: Numerical ranges of Hilbert space operators, vol. ${179}$, Cambridge University Press (2021)
			
			
			
			
		\bibitem{karaev2006berezin} Karaev, M. T.: Berezin set and Berezin number of operators and their applications. In the 8th Workshop on numerical ranges and numerical radii (WONRA-06), University of Bremen $\mathbf{14}$ (2006)
			
		\bibitem{karaev2013reproducing} Karaev, M. T.: Reproducing kernels and Berezin symbols techniques in various questions of operator theory. Complex Anal. Oper. Theory $\mathbf{7}$(4), 983–1018 (2013)
		
		
			
			
		\bibitem{maity2025convexity} Maiti, S. K., Sahoo, S., Chakraborty, G.: Convexity of the Berezin range of operators on $\mathcal{H}_\gamma(\mathbb{D})$. arXiv preprint arXiv:2505.24495 (2025)
			
		\bibitem{marcus1977constrained} Marcus, M., Andresen, P.: Constrained extrema of bilinear functionals. Monatsh. Math. $\mathbf{84}$(3), 219–235 (1977)

        \bibitem{martinez2007introduction} Mart{\'\i}nez-Avenda{\~n}o, R. A., Rosenthal, P.: An introduction to operators on the Hardy-Hilbert space, Springer (2007)
  
		
		
		
		\bibitem{paulsen2016introduction} Paulsen, V. I., Raghupathi, M.: An introduction to the theory of reproducing kernel Hilbert spaces, vol. ${152}$, Cambridge University Press (2016)
			
			
		\bibitem{sen2025berezin} Sen, A., Barik, S., Paul, K.: On the Berezin range of Toeplitz and weighted composition operators on weighted Bergman spaces. Complex Anal. Oper. Theory $\mathbf{19}$(8), 1--17 (2025)

         \bibitem{stojiljkovic25}
     Stojiljkovi\'{c}, V., Ba\c{s}aran, H., G\"{u}rdal, M.: On q-Berezin number inequalities in reproducing kernel Hilbert space. Filomat $\mathbf{39}$(28), 10129--10140 (2025)
   
			
		\bibitem{tapdigoglu2024some} Tapdigoglu, R.: Some results for weighted Bergman space operators via Berezin symbols. Math. Slovaca $\mathbf{74}$(2), 481–490 (2024)
			
		\bibitem{tsing1984constrained} Tsing, N.K.: The constrained bilinear form and the C-numerical range. Linear Algebra Appl. $\mathbf{56}$, 195–206 (1984)
			
		\bibitem{zamani2024berezin} Zamani, A., Sahoo, S., Tapdigoglu, R., Garaev, M.: A-Berezin number inequalities for $2\times2$ operator matrices. Bull. Malays. Math. Sci. Soc. $\mathbf{47}$(4), 114 (2024)		
		\end{thebibliography}
\end{document}